\documentclass[10pt]{amsart}

\usepackage{amsmath, amsthm, amssymb, graphicx, enumerate, everypage, datetime, getfiledate, wasysym} 
\usepackage{wrapfig}
\newcounter{citedtheorems}

\newcounter{theoremcounter}

\newtheorem{defn}[theoremcounter]{Definition}
\newtheorem{theorem}[theoremcounter]{Theorem}

\newtheorem*{theorem-m}{Theorem \ref{main-theorem}}

\newtheorem*{theorem-x}{Theorem}
\newtheorem*{theorem-abs1}{Theorem \ref{ind-theorem}}
\newtheorem*{theorem-abs2}{Theorem \ref{a23}}
\newtheorem*{theorem-abs3}{Theorem \ref{ind-new}}
\newtheorem*{theorem-abs4}{Theorem \ref{m1}}
\newtheorem{main-claim}[theoremcounter]{Main Claim}
\newtheorem{thm-lit}[citedtheorems]{Theorem}
\newtheorem{defn-lit}[citedtheorems]{Definition}
\newtheorem{fact-lit}[citedtheorems]{Fact}
\newtheorem{fact}[theoremcounter]{Fact}

\newtheorem{defn-claim}[theoremcounter]{Definition/Claim}

\newtheorem{concl}[theoremcounter]{Conclusion}
\newtheorem{conv}[theoremcounter]{Convention}
\newtheorem{claim}[theoremcounter]{Claim}

\newtheorem{rmk}[theoremcounter]{Remark}

\newtheorem{ntn}[theoremcounter]{Notation}
\newtheorem{disc}[theoremcounter]{Discussion}

\newtheorem{expl}[theoremcounter]{Example}

\newtheorem{qst}[theoremcounter]{Question}
\newtheorem{problem}[theoremcounter]{Problem}

\newcommand{\br}{\vspace{2mm}}

\newcommand{\eq}{\operatorname{eq}}

\newcommand{\ml}{\mathcal{L}}
\newcommand{\tlf}{\trianglelefteq}

\newcommand{\rn}{\operatorname{Range}}

\newcommand{\dom}{\operatorname{Dom}}

\newcommand{\acl}{\operatorname{acl}}
\newcommand{\dcl}{\operatorname{dcl}}

\newcommand{\lgn}{\operatorname{lg}}

\newcommand{\tpqf}{\operatorname{tp}_{\operatorname{qf}}}

\newcommand{\qftp}{\operatorname{tp}_{\operatorname{qf}}}

\newcommand{\tp}{\operatorname{tp}}

\newcommand{\xd}{\mathbf{d}}
\newcommand{\ocirc}{\astrosun}
\newcommand{\posl}{\operatorname{Pos}}
\newcommand{\negl}{\operatorname{Neg}}

\newcommand{\ii}{\mathbf{i}}

\newcommand{\ts}{\mathbf{S}}
\newcommand{\mck}{\mathcal{K}}

\newcommand{\rstr}{\upharpoonright}

\newcommand{\vp}{\varphi}

\newcommand{\ma}{\mathbf{a}}
\newcommand{\mb}{\mathbf{b}}

\newcommand{\trg}{T_{\mathbf{rg}}}

\newcommand{\GEM}{\operatorname{GEM}}

\newcommand{\xm}{\mathfrak{m}}

\newcommand{\tcb}{}

\newcommand{\xc}{\mathbf{c}}

\newcommand{{\xw}}{\mathbf{w}}
\newcommand{\xr}{\mathfrak{r}}

\newcommand{\mk}{\mathcal{K}}

\title{Shearing in some simple rank one theories}

\author{M. Malliaris and S. Shelah} 

\thanks{\emph{Thanks:} Research partially supported by NSF CAREER award 1553653, 
by an NSF-BSF award 
(NSF 2051825, BSF 3013005232), and ISF grant 1838/19. 
Paper 1221 in Shelah's list.}

\address{Department of Mathematics, University of Chicago, 5734 S. University Avenue, Chicago, IL 60637, USA}
\email{mem@math.uchicago.edu}

\address{Einstein Institute of Mathematics, Edmond J. Safra Campus, Givat Ram, The Hebrew
University of Jerusalem, Jerusalem, 91904, Israel, and Department of Mathematics,
Hill Center - Busch Campus, Rutgers, The State University of New Jersey, 110
Frelinghuysen Road, Piscataway, NJ 08854-8019 USA}
\email{shelah@math.huji.ac.il}
\urladdr{http://shelah.logic.at}

\begin{document}

\begin{abstract}  
Dividing asks about inconsistency along indiscernible sequences.   In order to study the finer structure of simple theories without 
much dividing, the authors recently introduced \emph{shearing}, which essentially asks about inconsistency along generalized 
indiscernible sequences.  Here we characterize the shearing of the random graph. We then use shearing to distinguish 
between the random graph and the theories $T_{n,k}$, the higher-order analogues of the triangle-free random graph.  
It follows that shearing is distinct from dividing in simple unstable theories, and distinguishes meaningfully between 
classes of simple unstable rank one theories. 
The paper begins with an overview of shearing, and includes open questions. 
\end{abstract}

\maketitle 

\vspace{3mm}

This paper is dedicated to Moshe Jarden, in honor of his many contributions to field arithmetic. 

One of the central points of contact between fields and model theory is the definition of forking/dividing, developed in the second author's 
book \cite{Sh:a}, which significantly abstracts the notions of algebraic independence in algebraically closed fields, or linear independence in 
vector spaces.  Informally, dividing asks about inconsistency along indiscernible sequences. 
 This definition has substantial explanatory power within stable theories, but outside of stable theories, it appears that also 
new ideas are needed.  Thus, in order to study the finer structure of simple theories without 
much dividing, the authors recently introduced \emph{shearing}, which essentially asks about inconsistency along generalized 
indiscernible sequences.   The aim of this paper is to further develop this very interesting definition. 

The paper \cite{MiSh:1149} proved a first separation theorem showing shearing 
is a priori useful for detecting differences in complexity in simple unstable theories. 
The proof that this notion is strictly weaker than dividing, and that it can be found in the random graph and indeed in 
any theory with the independence property,  
was deferred to the present work. Below, we carry out the characterization of shearing in the random graph announced there. 
We then show that shearing can distinguish between the random graph and the theories $T_{n,k}$, the higher-order analogues of the triangle-free random graph studied by Hrushovski. Perhaps this may open the door for a more careful structural analysis of simple rank one theories  
(whose structure is not visible to dividing) via shearing. 
Along the way we review what is known and record many natural questions. 

{We thank the anonymous referee for careful readings and very helpful reports.}

\section{Preliminaries}
\setcounter{theoremcounter}{0}

In this section we recall the central definition of \emph{shearing}, 
Definition \ref{e5}. The reader may prefer to look ahead, or to read on for the motivated development. 

The basic idea will be that whereas dividing corresponds to inconsistency along an indiscernible sequence, 
shearing corresponds to inconsistency along a generalized indiscernible sequence. One initial reason it might be hoped this would 
give some power is that generalized indiscernible sequences arise from the skeletons of $\GEM$-models.\footnote{Generalized Ehrenfeucht-Mostowski 
models -- informally, like Ehrenfeucht-Mostowski models except that we allow the index models to be reasonable expansions of 
linear orders, varying within an index model class $\mk$; 
for a recent exposition, see the early sections of \cite{MiSh:1124}, \cite{MiSh:1149}, \cite{Sh:E59}, \cite{Sh:300}.}  So 
from consistency or inconsistency along such sequences in a given $\GEM$-model one might hope to produce larger $\GEM$-models in which 
a given type was realized, or stayed omitted.  Going further, one might hope to show a difference in complexity between 
theories in this way, say, by showing that if $T_2$ has recurrent shearing for a certain kind of generalized indiscernible sequence and $T_1$ does not, 
then it would be possible to build a model of a theory interpreting both of them whose reduct to $T_1$ is quite saturated and whose reduct to 
$T_2$ is (say) not even $\aleph_1$-saturated. This was done in \cite{MiSh:1149}, using $\GEM$-models, and working through those 
proofs allowed us to arrive at the present definition of shearing. 
Once identified, the definition makes sense in any context, not only that of $\GEM$-models. 

We shall use three ideas from the model theory of generalized Ehrenfeucht-Mostowski models: 
that of an index model class $\mk$, that of a context $\xc$, and that of a generalized indiscernible sequence for such a class, which we call $\mk$-indiscernible. In various guises, these have long histories in model theory (a partial list might include 
\cite{EM}, \cite{Rabin}, \cite{Sh:a}, \cite{Sh:300}, \cite{scow2}, \cite{ghs}, \cite{MiSh:1124}, \cite{MiSh:1149}). 
To balance the demands of keeping the paper short but also reasonably self-contained, we point the reader to where these are clearly written down, and give here an English summary.

\emph{Index model class}: See \cite{MiSh:1149} Definitions 2.3, p. 4 and 2.9, p. 5. Briefly, 
$\mk$ is a class of linearly ordered models, closed under isomorphism but not necessarily an elementary class, 
in a signature expanding $\{ < \}$. We require that $\mk$ is universal \tcb{(closed under submodels and increasing chains)} and that for every $I \in \mk$ there is some $\aleph_0$-saturated 
$J  \in \mk$ extending it. (Since the class is not necessarily elementary, ``$\aleph_0$-saturated'' always abbreviates 
``$\aleph_0$-universal and $\aleph_0$-homogeneous''.) Finally, the class has to be Ramsey, in the usual sense of $\GEM$-models 
(or equivalently, by a theorem of Scow, in the sense of  Ne\v{s}et\v{r}il \cite{ns} and
of Kechris-Pestov-Todor\v{c}evi\'{c} \cite{kpt}; see Scow \cite{scow2} Theorem 4.31 and see \cite{Sh:a} Chapter VII, 
\cite{Sh:300} III. 1.5-1.15 pps. 327--332, \cite{Sh:E59}).  
Simple examples include: the class of linear orders, or the class of linear orders 
partitioned by countably many unary predicates $\{ P_n : n < \omega \}$, where note that requiring the countably many predicates to 
partition the domain (every element has a color; some colors may be empty) makes it not an elementary class. 

\emph{Context}: See \cite{MiSh:1149} Definition 2.12. Briefly, a context $\xc = (I, \mk)$ is a choice of index model class $\mk$ along with a 
choice of $I \in \mk$ which is not obviously trivial.  For example, taking $\mk$ to be the class of linear orders 
and asking for universality means $\mk$ contains lots of finite linear orders, but it is not so useful to choose one of them when asking about 
uniform inconsistency.  The various conditions essentially rule out analogues of this:  $I$ is closed under functions, if any;  
$I$ is not generated by any of its finite subsets; \tcb{$I$ is subject to certain mild technical conditions concerning algebraicity which will 
be given explicitly later on, and entail that 
things which only appear finitely many times in saturated extensions of $I$ do so for a reason}.  

A \emph{countable context} just means that, in addition, $I$ is countably infinite. 

\emph{$\mk$-indiscernible}:  See \cite{MiSh:1149} Section 3, or below. 
In a usual indiscernible sequence in a model $M$, 
indexed by a linear order $(I, <)$, if the finite sequences $\bar{t}, \bar{s}$ from $I$ have the same order-type 
[i.e., the same quantifier-free type in the language of order] 
then the corresponding tuples of elements they 
index have the same type in $M$. In general, if $I$ belongs to some $\mk$, then we may ask that if 
$\bar{t}, \bar{s}$ have the same quantifier-free type in $I$ then the corresponding tuples of elements they 
index have the same type in $M$. 
In the full definition, one more level of generality is implicit:

\begin{defn}[$\mk$-indiscernible sequence] \label{d:k-ind}
Let $\xc= (I, \mk)$. Suppose $N \models T$, $A \subseteq N$, and $f: {^{\omega >}I} \rightarrow 
{^{\omega>}N}$. For each $\bar{t} \in \dom(f)$, write $\bar{b}_{\bar{t}}$ for $f(\bar{t})$. Say $\mb = \langle \bar{b}_{\bar{t}} : 
\bar{t} \in {^{\omega >}I} \rangle$ is a $\mk$-indiscernible sequence over $A$ when: for all $k <\omega$, 
all $\bar{t}_0, \dots, \bar{t}_{k-1}$, $\bar{t}^\prime_0, \dots, \bar{t}^\prime_{k-1}$ from ${^{\omega >}I}$, if 
$\tpqf({\bar{t}_0}^\smallfrown {\bar{t}_1}^\smallfrown \dots ^\smallfrown {\bar{t}_{k-1}}, \emptyset, I) = 
\tpqf({\bar{t}^\prime_0} ~^\smallfrown {\bar{t}^\prime_1} ~^\smallfrown \dots ^\smallfrown {\bar{t}^\prime_{k-1}}, \emptyset, I)$ then 
$\lgn({\bar{a}_{\bar{t}_i}}) = \lgn(\bar{a}_{\bar{t}^\prime_i})$ for $i < k$, and 
\[ \tp({\bar{a}_{\bar{t}_0}} ~^\smallfrown \dots ^\smallfrown {\bar{a}_{\bar{t}_{k-1}}}, A, N) = 
\tp({\bar{a}_{\bar{t}^\prime_0}} ~ ^\smallfrown \dots ^\smallfrown {\bar{a}_{\bar{t}^\prime_{k-1}}}, A, N). \] 
\end{defn}
Observe that in \ref{d:k-ind}, requirements on coherence of these maps between arities are conspicuously absent. For instance, 
$f(t_1 t_2)$ isn't required to be the concatenation of $f(t_1)$ with $f(t_2)$, \tcb{and so some sequences can map to $\emptyset$.} 

\begin{rmk} \label{rmk:r-seq}
\tcb{
Here is a useful construction which fits the definition just given. 
Choose a particular quantifier-free $k$-type $\xr$ of some sequence $\bar{t}$ of elements of $I$. 
Choose some $\aleph_0$-saturated $J \in \mk$ with $I \subseteq J$. Consider the set  $\xr(J) = \{ \bar{s} : \tpqf(\bar{s}, \emptyset, J) = \tpqf(\bar{t}, \emptyset, I) \}$. Consider a map $f$ taking each $\bar{s}$ in $\xr(J)$ to some $\bar{a}_{\bar{s}}$ in $M$, 
so that the types of any two such sequences in $M$ are the same. 
Extend $f$ to ${^{\omega> }J}$ by sending anything outside $\xr(J)$ to the empty set $($or to a suitable single fixed element$)$. 
Then the image of $f$ is a $\mk$-indiscernible sequence.  A further special case of this is starting with a 
generalized indiscernible sequence $\langle a_t : t \in I \rangle$ indexed 
by some $I \in \mk$, and looking at tuples indexed by a fixed quantifier-free type. 
In the present paper, $\mk$-indiscernible sequences of the forms just described, for some $\xr$ to be specified in each case, 
will generally suffice.} 
\end{rmk}

\begin{expl} \label{e:236} 
\emph{The following may help understanding, as will be explained.} 
\end{expl}

The sequence $\langle \bar{\sigma}^M(\bar{a}_{\bar{t}}) : \bar{t} \in \xr(I) \rangle$ is indiscernible in $M$ when: 
\begin{enumerate}
\item[(a)] $I \in \mk$ is an index model, 

\item[(b)] $M$ is a model in the signature $\tau$, 

\item[(c)] $\langle \bar{a}_t : t \in I \rangle$ is an indiscernible sequence in $M$, 

\item[(d)] $\bar{a}_{\bar{t}} = \langle a_{t_\ell} : \ell < n \rangle$ when $\bar{t} \in {^n I}$. 

\item[(e)] $\bar{\sigma}(x_{[n]}) = \langle \sigma_\ell(\bar{x}_{[n]}) : \ell < k \rangle$ where $\sigma_\ell$ is a function symbol in $\tau$, 
or a term.\footnote{Notation: $\bar{x}_{[n]}= \langle x_\ell : \ell < n \rangle$.} 

\item[(f)] $\bar{\sigma}^M(\bar{a}_{\bar{t}}) = \langle \sigma^M_\ell (\bar{a}_{\bar{t}}) : \ell < k \rangle$

\item[(g)] $r(\bar{x}_{[n]}) \in \{ \tpqf(\bar{t}, \emptyset, I) : \bar{t} \in {^n I} \}$. 

\item[(h)] and $\xr(I) = \{ \bar{t} \in {^n I} : \bar{t} $ realizes the quantifier-free type $r$ in $I$ $\}$. 

\end{enumerate}
In more detail, 
the reader may wonder about how the Ramsey property may interact with an arbitrary $\mk$-indiscernible sequence of the form $\langle \bar{b}_{\bar{t}} : \bar{t} \in \xr(I) \rangle$, since the Ramsey property deals with sequences indexed by singletons.  To explain this, recall that the Ramsey property says the following. 

\begin{defn} We say the class $\mk$ is Ramsey when: given any 

a) $J \in \mk$ which is $\aleph_0$-saturated

b) model $M$, and 

c) sequence $\mb = \langle \bar{b}_t : t \in J \rangle$ of finite sequences from $M$, 

with the length of $\bar{b}_t$ determined by $\tpqf(t, \emptyset, J)$, 

\noindent there exists\footnote{recall that a template $\Psi$ is called proper for $I$ if there exists a (generalized) Ehrenfeucht-Mostowski model of the form $M = \operatorname{GEM}(I, \Psi)$, i.e. these instructions are coherent and give rise to a model over an $I$-indexed skeleton.} a template $\Psi$ which is proper for $\mk$ such that:

\noindent i) $\tau(M) \subseteq \tau(\Psi)$,

\noindent ii) $\psi$ reflects $\mb$ in the following sense:

for any $s_0, \dots, s_{n-1}$ from $J$, 

any $\theta = \theta(x_0, \dots, x_{m-1})$ from $\ml(\tau(M))$, 

and any $\tau(M)$-terms $\sigma_\ell(\bar{y}_0, \dots, \bar{y}_{n-1})$ for $\ell = 0, \dots, m-1$, 

\underline{if} $M \models \theta[\sigma_0(\bar{b}_{t_0}, \dots, \bar{b}_{t_{n-1}}), \dots, \sigma_{m-1}(\bar{b}_{t_0}, \dots, \bar{b}_{t_{n-1}})]$ 

for every $t_0, \dots, t_{n-1}$ realizing $\tpqf({s_0}^ \smallfrown ~\cdots ~^\smallfrown {s_{n-1}}, \emptyset, J)$ in $J$, 

\underline{then} $\operatorname{GEM}(J, \Psi) \models 
 \theta[\sigma_0(\bar{a}_{t_0}, \dots, \bar{a}_{t_{n-1}}), \dots, \sigma_{m-1}(\bar{a}_{t_0}, \dots, \bar{a}_{t_{n-1}})]$

where $\langle \bar{a}_s : s \in J \rangle$ denotes the skeleton of $\operatorname{GEM}(J, \Psi)$. 
\end{defn}

\textcolor{black}{
In this paper, we will be interested in inconsistency along generalized indiscernible sequences: see Definition \ref{e5} below, and observe there in item (5) that the inconsistency is always within a particular $\xr(I)$.  So let us verify the following.
}

\textcolor{black}{
Let $\mb = \langle \bar{b}_{\bar{t}} : \bar{t} \in \xr(J) \rangle$ be a $\mk$-indiscernible sequence in the sense of Remark \ref{rmk:r-seq}, where 
$J \in \mk$ is $\aleph_0$-saturated. In particular, this sequence may be the range of a function whose domain is the finite sequences of elements of $J$, but which sends any sequence not in $\xr(J)$ to $\emptyset$. }

\textcolor{black}{
Suppose that for some $\vp$, the set of formulas $\{ \vp(\bar{x}, \bar{b}_{\bar{t}}) : \bar{t} \in \xr(J) \}$ is contradictory.  
Let $M_0 = \GEM(J, \Phi)$ be a generalized EM model with skeleton $\langle \bar{a}_s : s \in J \rangle$.  Observe that in $M_0$, for any 
$\bar{t} \in \xr(J)$, the expression $\bar{a}_{\bar{t}}$ makes sense, and because this is an actual skeleton of an actual GEM model, writing 
$\bar{t} = \langle t_0, \dots, t_{k-1} \rangle$, necessarily 
$\bar{a}_{\bar{t}} = {\bar{a}_{t_0}}~^\smallfrown ~\cdots ~^\smallfrown {\bar{a}_{t_{k-1}}}$. }

\begin{claim}
\textcolor{black}{
In the context just given, 
there is a template $\Psi \geq \Phi$ and $N = \GEM(J, \Psi)$ an extension of $M_0$ with the same skeleton $\langle \bar{a}_s : s \in J \rangle$, and a finite sequence of 
$\ell(\bar{a}_s)$-ary functions $\bar{F} = \langle F_0, \dots, F_{\ell(\bar{b}_{\bar{t}})-1} \rangle$ in $\tau(\Psi)$, 
so that in $N$,  the set of formulas $\{ \vp(\bar{x}, \bar{F}(\bar{a}_{\bar{t}})) : \bar{t} \in \xr(J) \}$ is contradictory.
}
\end{claim}

\begin{rmk}
This claim does not assert that the inconsistency is witnessed by the skeleton, just that it is witnessed in the $\GEM$-model. 
\end{rmk}

\begin{proof}
\textcolor{black}{Let $M$ be an elementary extension of $M_0$ which contains $\mb$.  Expand the language to include the function symbols $\bar{F}$ (the $\sigma$'s of example \ref{e:236} above). 
Interpret them so that for each $\bar{t} \in \xr(J)$, $\bar{F}(\bar{a}_{\bar{t}}) = \bar{b}_{\bar{t}}$. }

\textcolor{black}{Since $\{ \vp(\bar{x}, \bar{b}_{\bar{t}}) : \bar{t} \in \xr(J) \}$ is contradictory, there must be some finite set witnessing it, say, 
$\{ \bar{r}_0, \dots, \bar{r}_{\ell-1} \} \subseteq \xr(J)$ so that $\{ \vp(\bar{x}, \bar{b}_{\bar{r}_0}), \dots, \vp(\bar{x}, \bar{b}_{\bar{r}_{\ell-1}}) \}$ is 
contradictory. Since $\mb$ is $\mk$-indiscernible, this depends only on the quantifier-free type of 
$ \bar{r}_0 ~^\smallfrown \cdots ^\smallfrown~ \bar{r}_{\ell-1}$.  Let $\theta$ be a formula so that $\theta[ \bar{b}_{\bar{r}_0}, \dots, \bar{b}_{\bar{r}_{\ell-1}}]$ expresses  that $\{ \vp(\bar{x}, \bar{b}_{\bar{r}_0}), \dots, \vp(\bar{x}, \bar{b}_{\bar{r}_{\ell-1}}) \}$ is contradictory. }

\textcolor{black}{Let $N = \GEM(J, \Psi)$ be given by the Ramsey property.  Then for any sequence of tuples 
$ \bar{s}_0 ~^\smallfrown \cdots ^\smallfrown~ \bar{s}_{\ell-1}$ 
of the same quantifier-free type as $ \bar{r}_0 ~^\smallfrown \cdots ^\smallfrown~ \bar{r}_{\ell-1}$ in $J$, 
 $\theta[\bar{F}(\bar{a}_{\bar{s}_0}), \dots, \bar{F}(\bar{a}_{\bar{s}_{\ell-1}})]$ will hold in $N$,  which proves the claim. } 
\end{proof}

\textcolor{black}{In other words, even though the generalized indiscernible sequences we use may really only focus on things indexed by tuples, 
in the cases where we apply the Ramsey property or where we assume that instances of inconsistency arise inside a $\GEM$-model, 
the Ramsey property is only ever applied to skeleta or other sequences indexed by singletons. 
}

\begin{ntn}  \label{ntn-exp} When $I_0$ is a set and $I_0 \subseteq J \in \mk$, writing  
$J[I_0]$ means $J$ expanded by constants for the elements of $I_0$, and likewise for $J[\bar{s}]$ when $\bar{s} \subseteq J$ is a sequence.
\end{ntn}

We now arrive at our central definition. 

\begin{defn}[Shearing, \cite{MiSh:1149} Definition 5.2] \label{e5}  
Suppose we are given a context $\xc$, a theory $T$, $M \models T$, $A \subseteq M$, and a formula $\vp(\bar{x}, \bar{c})$ of the language of $T$ with 
$\bar{c} \in {^{\omega>} M}$. 
We say that 
\[ \mbox{ \emph{the formula $\vp(\bar{x}, \bar{c})$ shears over $A$ in $M$ for $(I_0, I_1, \xc)$} } \]
when there exist a model $N$, a sequence $\mb$ in $N$, enumerations $\bar{s}_0$ of $I_0$ and $\bar{t}$ of $I_1$, and an $\aleph_0$-saturated $J \supseteq I$ such that:
\begin{enumerate}
\item $I_0 \subseteq I_1$ are finite subsets of $I$
\item $M \preceq N$
\item $\mb = \langle \bar{b}_{\bar{s}} : \bar{s} \in {^{\omega >}(J[I_0])} \rangle $ is $\mk$-indiscernible in $N$ over $A$
\item $\bar{c} = \bar{b}_{\bar{t}}$, and 
\item the set of formulas 
\[ \{ \vp(\bar{x}, \bar{b}_{\bar{t}^\prime}) : \bar{t}^\prime \in {^{\lgn(\bar{t})}(J)}, 
\tpqf(\bar{t}^\prime, \bar{s}_0, J) = \tpqf(\bar{t}, \bar{s}_0, I) \} \]
is contradictory.
\end{enumerate}
\end{defn}

Let us look at this definition a little more closely. From the second paragraph of the section, 
the reader may anticipate that it will be useful to be able to iterate any inconsistency we may find. 
Perhaps we start by observing that 
in some larger $J \supseteq I$, the tuples satisfying the same quantifier-free type as our given $\bar{t}$ can form indices for a generalized 
indiscernible sequence along which some $\vp$ is inconsistent. If we fix one of these instances, say $\bar{s}_0$, can this happen again over $\bar{s}_0$? 
That is, can we find some $\bar{t}_1$ from $I$ so that in some larger $J \supseteq I$, the tuples satisfying the quantifier-free type 
of $\bar{t}_1$ over $\bar{s_0}$ form indices for a generalized indiscernible sequence along which our $\vp$ is again inconsistent? 
Does this stop after finitely many steps?  This motivates the definition of \emph{$\xc$-superstability}, which was key to \cite{MiSh:1149}, 
and explains the appearance of a finite ``$I_0$'' in Definition \ref{e5}. 

\begin{rmk}
\tcb{Regarding Definition $\ref{e5}$, we may refer to this concept in different ways in the rest of the paper, such as 
``$\vp(\bar{x}; \bar{c})$ $(I_0, I_1)$-shears over $B$''; the formulation in $\ref{e5}$ inserts a verb between the 
$\vp$ and the $(I_0, I_1)$ to make parsing easier.}
\end{rmk}

We state here the local definition of ``$\xc$-superstable'' just for relational languages, which suffices for the present paper. Note that 
countability of the context $\xc$ is used in an essential way, as $I$ is written as the union of an increasing chain of finite sets.  

\begin{defn}[\cite{MiSh:1149} Definition 6.2, local version] \label{m5a}  Let $\xc$ be a countable context and $\Delta$ a set of formulas. 
We say $(T, \Delta)$ is \emph{unsuperstable} for $\xc$ when there are: 

\begin{enumerate}
\item[(a)] 
an increasing sequence of nonempty finite sets $\langle I_n : n <\omega \rangle$ with 
$I_m \subseteq I_{n} \subseteq I$ for $m < n <\omega$ and $\bigcup_n I_n = I$, 
which are given along with a choice of enumeration $\bar{s}_n$ for each $I_n$ 
where $\bar{s}_n \tlf \bar{s}_{n+1}$ for each $n$ 

\item[(b)]  
an increasing sequence of nonempty, possibly infinite, sets $B_n \subseteq B_{n+1} \subseteq \mathfrak{C}_T$ in the monster 
model for $T$, with $B := \bigcup_n B_n$

\item[(c)] and a partial type $p$ over $B$, 
\end{enumerate}

\noindent such that for each $n$, for some formula $\vp(\bar{x}, \bar{c})$ from $p \rstr B_{n+1}$ where 
$\vp(\bar{x}, \bar{y}) \in \Delta$,  we have that 
\[ \vp(\bar{x}, \bar{c}) \mbox{\hspace{4mm}} (I_n, I_{n+1})\mbox{-shears over }B_n. \]
\end{defn}

\begin{defn} Continuing Definition $\ref{m5a}$, for a countable context $\xc$, 
\begin{enumerate}
\item we may write $(T, \vp)$ to mean $(T, \{ \vp \})$.
\item We say $T$ is $\xc$-superstable if $(T, \Delta)$ is $\xc$-superstable where $\Delta$ is the set of all formulas in the language. 
\item We say $T$ is $\xc$-stable if $(T, \vp)$ is $\xc$-superstable for every $\vp$ in the language. 
\end{enumerate}
\end{defn}

\begin{disc}
\emph{
These investigations into the fine structure of forking highlight a longstanding terminological point, which hopefully should not 
cause confusion if explicitly pointed out.  Namely, in the classical case, 
both ``superstable'' and ``supersimple'' \tcb{are connected to} ``$\kappa(T) = \aleph_0$.''  In Definition \ref{m5a}, we would also have 
been justified in using ``supersimple''.  It seems to us that stable is the right one to use for various reasons, so hopefully this 
does not confuse the reader 
in a sentence like ``a theory is simple if there is \emph{some} context $\xc$ for which it is $\xc$-stable.'' 
}
\end{disc}

For context and easy quotation, 
the next local summary theorem includes some already known facts, and some new results proved below.  
We assume our theories $T$ are all complete.

\begin{theorem}[Local shearing] \label{t:local} For now all contexts are countable. 
\begin{enumerate}
\item Dividing implies shearing. 
\item Shearing does not imply dividing, i.e., it is strictly weaker than dividing (i.e. for some relevant context $\xc$). 
\item If $\vp$ is a stable formula in the theory $T$, then $(T, \vp)$ is $\xc$-superstable for \emph{every} countable context $\xc$. 
\item If $\vp$ is an unstable formula in the theory $T$, then $(T, \vp)$ is $\xc$-unsuperstable for some countable context $\xc$. 
$($So the previous item is a characterization of stable formulas.$)$ 
\item In the previous item, we can take $\xc$ to be a countable context from the class $\mk$ of linear orders, that is, $(T, \vp)$ is unstable if and 
only if it is $\xc$-unsuperstable for a countable context chosen from the class of infinite linear orders. 
\item If $(T, \vp)$ is $\xc$-superstable for \emph{some} countable context $\xc$, then $\vp$ is simple, 
i.e. it does not have the tree property in the theory $T$.  $($So the natural focus of shearing is in some sense within simplicity.$)$ 
\item 
More precisely, in the previous item, the following are equivalent: 
 (a) $\vp$ is simple in $T$ and\footnote{Recall that $\kappa_{\operatorname{loc}}(T, \vp) = \aleph_0$ means that 
 every $\vp$-type does not fork over a finite set.}
  $\kappa_{\operatorname{loc}}(T, \vp) = \aleph_0$, and (b) there is some 
countable context $\xc$ for which $(T, \vp)$ is $\xc$-superstable.
\end{enumerate}
\end{theorem}

\begin{proof}
(1) See \cite{MiSh:1149} Claim 5.8. 

(2) It was stated in \cite{MiSh:1149} (see the end of \S 5) that this would be proved in the present paper, which it is indeed: 
we shall show that there is nontrivial shearing for the random graph and also for the theories $T_{n,k}$ for 
$n > k \geq 2$, theories which have dividing only for equality.  

(3)-(4) See Conclusion \ref{c:unstable} below. 

(5) Conclusion \ref{c:unstable} and Example \ref{xyz-lin} below.  But this should be read with caution, see \ref{disc:17}. 

(6)-(7) See \cite{MiSh:1149} Lemma 9.5 and Theorem 9.6. 
\end{proof}

\begin{disc} \label{disc:17}
\emph{
Caution: just because dividing implies shearing, one should not jump to conclusions about linear orders. As noted in \ref{c:unstable}(5), 
the random graph is $\xc$-unsuperstable for contexts coming from $\mk$ the class of linear orders, 
see \ref{xyz-lin} below. (This uses an indexing by pairs.)  However the random graph is $\xc$-superstable for some contexts coming from 
expansions of linear orders by predicates, see \ref{xyz-nonlin}. 
}
\end{disc}

\begin{disc}
\emph{
In \ref{c:unstable}(7), we might want a characterization of ``$(T, \vp)$ is simple'' and wonder why the apparently 
extra assumption  ``$\kappa_{\operatorname{loc}}(T, \vp) = \aleph_0$''   appears. The reason is that 
in \ref{c:unstable}, we are using countable contexts while making no 
assumptions on cardinality of the theory, so there is an extra point about the length of the shearing chain. 
\tcb{That is, we may wish to define $\kappa_\xc(T)$ for a context $\xc$ and a complete theory $T$ to be the minimal $\kappa$ such that there are no $A_\alpha \subseteq \mathfrak{C}_T$ for $\alpha < \kappa$ such that $\alpha < \beta < \kappa$ implies $A_\alpha \subseteq A_\beta$, 
and a type $p \in \ts(A_\kappa, \mathfrak{C}_T)$ such that $p \rstr A_{\alpha +1}$ $\xc$-shears over $A_\alpha$, but then we need an 
analogous condition on the $I$'s as in \ref{m5a}(a); and if the language is larger, we may get long chains not coming from a single formula.}  
So to fully capture simplicity, it would be natural to consider larger contexts to deal with the analogue of arbitrarily large $\kappa(T)$. 
If the reader is inspired by this remark, \cite{MiSh:1149} \S 10 is a beginning. 
}
\end{disc}

\vspace{5mm}

\section{Background}
\setcounter{theoremcounter}{0}

In later sections we will use shearing to distinguish between various theories which have no dividing in the usual sense (other than 
that coming from equality). In this expository section we review why this is so.  For background on simple theories, see the 
survey \cite{GIL}.   

\begin{defn}
For each $n \geq 2$, let $T_{n,1}$ be the theory of the generic $K_{n+1}$-free graph.  For $n > k \geq 2$, 
let $T_{n,k}$ be the theory of the generic $(n+1)$-free $(k+1)$-hypergraph.
\end{defn}

So, in this notation, $T_{3,2}$ is the theory of the generic tetrahedron-free three-hypergraph. For $k \geq 2$, i.e. for the case where the edge is 
really a hyperedge, these theories are simple unstable with trivial forking, as 
shown by Hrushovski \cite{hrushovski1}.  For context, we begin by reviewing why the graph versions of these theories 
(the generic triangle-free graph and its relatives) \emph{do} have a lot of dividing. 

\begin{fact} \label{fact15}
$T_{n,1}$ is not simple.  
\end{fact}

\begin{proof}[Proof sketch]
First consider the triangle-free random graph, in our notation $T_{2,1}$. Let $\vp(x,y,z) = R(x,y) \land R(x,z)$. Consider in 
any model of $T_{2,1}$ an infinite sequence of pairs 
$\ma = \langle \bar{a}_i : i < \omega \rangle$ where each $\bar{a}_i = a^0_i ~a^1_i$, and such that $R(a^0_i, a^1_j)$ for all $i\neq j$ but for 
all $i,j$ 
$\neg R(a^0_i, a^0_j)$ and $\neg R(a^1_i, a^1_j)$. Informally, the sequence $\ma$ is a bipartite graph which is the complement of a matching. 
Then the sequence 
\[ \{ \vp(x, a^0_i, a^1_i) : i < \omega \} \]
has the property that each formula is individually consistent, but the sequence is 2-inconsistent, since if $i \neq j$ then 
any element satisfying $\{ R(x,a^0_i),  R(x,a^1_j) \}$ would form a triangle.   Moreover, it is easy to see that for any $\bar{a}_i$ in $\ma$, 
we can construct a sequence \tcb{conjugate (i.e., isomorphic)} to $\ma$ which is indiscernible over $\bar{a}_i$, and continuing in this way we may 
construct the tree property for $\vp$, 
showing that $\vp$ is not simple. 

It is easy to extend this idea to $n >2$ using $\vp = R(x,y_0) \land \cdots \land R(x,y_{n-1})$, replacing $\ma$ by a sequence of $n$-tuples 
$\bar{a}_i = a^0_i \cdots a^{n-1}_i$ and specifying that $R(a^{s}_i, a^{t}_j)$ holds there if and only if $(s \neq t) \land (i \neq j)$. 
\end{proof}

When we consider hypergraphs instead of graphs the situation is quite different. 

\begin{fact}[Hrushovski c. 2002, see \cite{hrushovski1}]  \label{t:trd}
For $n>k\geq 2$, $T_{n,k}$ is simple unstable with only trivial dividing.
\end{fact}

\begin{proof}[Proof sketch]
We sketch the proof for $T = T_{3,2}$, the tetrahedron-free 3-hypergraph, since this extends naturally to larger arities but 
at a notational cost. 
Suppose there were some formula $\vp(x,a^0,\dots, a^{\ell-1})$, some $m > 1$, and some indiscernible sequence 
$\ma  = \langle \bar{a}_i : i <\omega \rangle$ such that $\bar{a}_0 = \langle a^0_0, \dots, a^{\ell-1}_0 \rangle = \langle a^0, \dots, a^{\ell-1} \rangle$ and 
\begin{equation}
\label{star} \{ \vp(x,a^0_i,\dots, a^{\ell-1}_i)  : i < \omega \} 
\end{equation}
is $m$-inconsistent. \tcb{Consider first the case where $\ell(\bar{x})=1$.}
Let's consider $\ma$ as being arranged so that each $\bar{a}_i$ is a column, and each $\langle a^{s}_i : i < \omega \rangle$ 
is a a row.   We don't assume anything about how the edges hold on $\ma$, but there are some constraints, e.g. 
because $\ma$ is indiscernible and exists in a model of $T$, 
the edge $R(x,y,z)$ cannot hold on any three distinct elements in any row of $\ma$, otherwise (by indiscernibility) the row would 
contain a tetrahedron. 
By quantifier elimination, and without loss of generality ignoring the trivial forking coming from equality,  
we may assume $\vp$ is a boolean combination of instances of $R(x;y,z)$.   

Now if  (\ref{star}) 
is inconsistent, there must be some tetrahedron which appears. In particular, there must be elements $b,c,d$ 
which occur in $\ma$ with the following three properties: first, the quantifier-free type of $\ma$ implies $R(b,c,d)$;  
second, $ \{ \vp(x,a^0_i,\dots, a^{\ell-1}_i)  : i < \omega \}  \vdash \{ R(x,b,c), R(x,b,d), R(x,c,d) \}$; and third, 
because each individual formula in $(\ref{star})$ is consistent,  
$b,c,d$ are not all in one column of $\ma$. 
But $(\ref{star})$ can only imply instances of formulas all of whose parameters occur in the 
same column -- look at the definition of (\ref{star}) and notice that none of its formulas include parameters from 
distinct columns.  [This is the crucial difference in having an edge of higher arity than 2.]  So this contradiction can never arise. 

\tcb{Observe that if $\ell(\bar{x}) > 1$, any tetrahedron arising must necessarily involve an edge on one of the $x_i$'s and parameters
$b$, $c$ from two distinct columns (as each individual instance of $\vp$ must remain consistent) 
and so a similar analysis applies. Finally, observe that if $\ma$ is indiscernible over some set 
$A$ rather than just over the empty set, the only new case is when $\ell(x) = 1$ and one of 
$b,c,d$ in the above argument may come from $A$, so we again 
reduce to the problem of asserting an edge across two distinct columns.}

This completes the (sketch of the) proof. 
\end{proof}

\br

\section{Analysis of the random graph} 
\setcounter{theoremcounter}{0}

\begin{conv} 
Reminder: all contexts are countable.
\end{conv}

The theory of the random graph, $\trg$, is central to the picture. In this section we analyze it carefully and give a complete characterization 
of countable contexts $\xc$ for which $\trg$ is $\xc$-superstable, Theorem \ref{t:circle}.   This result was announced in \cite{MiSh:1149}, along with Definition 3.12 and the motivating example preceeding it; although it isn't strictly necessary, we repeat the example here for clarity. 

The direct route to Theorem \ref{t:circle} is via Definition \ref{d:circ}, Claim \ref{claim15a} and Claim \ref{c:west}. 
There are two discussions and a claim which are mainly explanatory: Discussion \ref{sub:ex}, Discussion 
\ref{disc:x2} and Claim \ref{claim15a}.
 
 \subsection{A motivating example} \label{sub:ex}
Fix for awhile $\xc = (I, \mk) = (I_\xc, \mk_\xc)$ some countable context and 
we shall investigate how $\xc$-shearing may arise for $\trg$ inside a $\GEM$-model and try to find a characteristic property 
of a countable context which explains how such shearing occurs.
Consider $M = \GEM(I, \Phi)$, where $\Phi \in \Upsilon[\trg]$ thus $(M, R^M) \models \trg$, 
and let $p \in \ts(M \rstr \tau(\trg))$ be a nonalgebraic type. Fix $J$ such that $I \subseteq J \in \mk$ and $J$ is $\aleph_0$-saturated. 
Let $N  = \GEM(J, \Phi)$. Since all the contexts we will consider are well behaved (see \cite{MiSh:1149} Convention 2.7, or take the following as 
a provisional definition of good behavior)  we have 
$M \preceq N$, so we will identify the sequence 
$\langle \bar{a}_t : t \in I \rangle$ which generates $M$ with a subsequence of $\langle \bar{a}_t : t \in J \rangle$. 

By quantifier elimination, we may suppose $p$ is equivalent to  
$\{ R(x,b_\alpha)^{\ii_\alpha} \land x \neq b_\alpha ~:~ \alpha < \kappa \}$ for some infinite $\kappa$, where each 
$\ii_\alpha \in \{ 0, 1 \}$.  As $M$ is generated by $\{ \bar{a}_t : t \in I \}$, each $b_\alpha$ 
may be written as 
$\sigma^M_\alpha(\bar{a}_{\bar{t}_\alpha})$ for some $\tau(\Phi)$-term $\sigma_\alpha$ and some \tcb{finite sequence $\bar{t}_\alpha$ 
from $I$}.\footnote{\tcb{Since we will soon take all possible representations, we don't require $\bar{t}_\alpha$ to be strictly 
increasing here; when we fix a quantifier-free type of $\bar{t}_\alpha$ later, the order type will become determined.} }  
This representation may not be unique;  
there is no harm in choosing our enumeration to include all representations.
So without loss of generality, for some $\kappa = \kappa + |\tau(\Phi)|$, 
\begin{equation} \label{eq:p}
p(x) \equiv \{ R(x, \sigma^M_\alpha(\bar{a}_{\bar{t}_\alpha}))^{\ii_\alpha} \land 
x \neq \sigma^M_\alpha(\bar{a}_{\bar{t}_\alpha}) : \alpha < \kappa \} 
\end{equation}  where for each $\alpha < \kappa$ and $\bar{t}$ \tcb{$\in I^{<\omega}$}
 if $\sigma^M(\bar{a}_{\bar{t}}) = \sigma^M_\alpha(\bar{a}_{\bar{t}_\alpha})$, 
then 
\begin{equation} \label{eq:r}
\mbox{ for some $\beta < \kappa$, $\sigma_\beta = \sigma$ and $\bar{t}_\beta = \bar{t}$. }
\end{equation}
This may increase the length of the enumeration, but will not change the size of the type in $\tau(\trg)$.  
[From the point of view of $M \rstr \tau(\trg)$,
we may appear to list some (say)  
$R(x,b_\alpha)^{\ii_\alpha} \land x \neq b_\alpha$ many times, because we have listed an instance for each way of 
writing $b_\alpha$ in $M$ in terms of the skeleton. From the point of view of 
our enumeration, which has access to $\tau(\Phi)$, for each $b_\alpha$ there are potentially $|I| + |\tau(\Phi)|$ such 
representations.]

Recalling our fixed $\aleph_0$-saturated $J$ extending $I$ and its associated 
$N = \GEM(J, \Phi)$, we ask about potential $\xc$-shearing. 
Working in $N$, consider the set of formulas 
\begin{align*} \label{eq:q}
q(x) = q_{I_0}(x) =  & \{ R(x, \sigma^N_\alpha(\bar{a}_{\bar{u}}))^{\ii_\alpha} 
\land x \neq \sigma^M_\alpha(\bar{a}_{\bar{u}}) : 
 \alpha < \kappa, \\ & \mbox{ \hspace{5mm} } \tpqf(\bar{u}, I_0, J) = 
\tpqf(\bar{t}_\alpha, I_0, I) \}. 
\end{align*} 
To show $q(x)$ is consistent, it would suffice to check that whenever 
\begin{equation} 
R(x, \sigma^N_\alpha(\bar{a}_{\bar{v}}))^{\ii_\alpha}  \in q \mbox{ and } R(x, \sigma^N_\beta(\bar{a}_{\bar{w}}))^{\ii_\beta}  \in q 
\end{equation}
we have that {if} $\sigma^N_\alpha(\bar{a}_{\bar{v}}) = \sigma^N_\beta(\bar{a}_{\bar{w}})$ {then} $\ii_\alpha = \ii_\beta$. 
Suppose this fails, i.e., for some suitable $\alpha, \beta, \bar{v}_*, \bar{w}_*$ which we fix for awhile, 
$q$ contains the contradictory formulas 
\begin{equation}
\label{eq:as1} R(x,\sigma^N_\alpha(\bar{a}_{\bar{v}_*}))  \mbox{ and } \neg R(x,\sigma^N_\beta(\bar{a}_{\bar{w}_*})). 
\end{equation}
In other words, 
\begin{equation} 
\label{eq:as}
\sigma^N_\alpha(\bar{a}_{\bar{v}_*}) = \sigma^N_\beta(\bar{a}_{\bar{w}_*}) = {b} 
\mbox{ but } \ii_\alpha \neq \ii_\beta
\mbox{ (here w.l.o.g. $\ii_\alpha = 1$, $\ii_\beta = 0$). }
\end{equation} 
Informally, what has happened is that in $N$, there is a ``positive line'' 
\[ \posl = \{ \sigma^N_\alpha(\bar{a}_{\bar{v}}) :  \tpqf(\bar{v}, I_0, J) = \tpqf(\bar{t}_\alpha, I_0, I) \} \subseteq \dom(N) \]
and a ``negative line'' 
\[ \negl = \{ \sigma^N_\beta(\bar{a}_{\bar{w}}) :  \tpqf(\bar{w}, I_0, J) = \tpqf(\bar{t}_\beta, I_0, I) \} \subseteq \dom(N) \]
and the problem is that 
\begin{equation} \label{eq:pl}
\posl \cap \negl \neq \emptyset, \mbox{ witnessed by } b = \sigma^N_\alpha(\bar{a}_{\bar{v}_*}) = \sigma^N_\beta(\bar{a}_{\bar{w}_*}). 
\end{equation}
However, it is also important to notice that both ``lines'' have ``points from $I$'', 
and that these are not the point(s) of intersection:\footnote{If $p$ is a complete type in $M$, 
then we will have the stronger statement that ``the restrictions of $\posl$ and $\negl$ to $I$ have no intersection,'' i.e. 
$\{ \sigma^N_\alpha(\bar{a}_{\bar{v}}) :  \tpqf(\bar{v}, I_0, I) = \tpqf(\bar{t}_\alpha, I_0, I) \} \cap  
\{ \sigma^N_\beta(\bar{a}_{\bar{w}}) :  \tpqf(\bar{w}, I_0, I) = \tpqf(\bar{t}_\beta, I_0, I) \} = \emptyset$, but we do not need this here.}
\begin{equation} \label{eq:7} \sigma^N_\alpha(\bar{a}_{\bar{t}_\alpha}) \neq \sigma^N_\beta(\bar{a}_{\bar{t}_\beta}) 
\end{equation} 
else our original $p$ would be inconsistent. (So both $\posl$ and $\negl$ have size $\geq 2$.) 

Definition \ref{d:circ} abstracts the key property of $\xc$ behind this picture. Towards this, 
observe that writing $\xr_\alpha = \tpqf(\bar{t}_\alpha, I_0, I)$, we have that ``$\sigma^N_\alpha(\bar{a}_{\bar{v}}) = \sigma^N_\alpha(\bar{a}_{\bar{u}})$'' is an equivalence relation on ${\xr_\alpha}(J)$ (asserting that $\bar{v}, \bar{u}$ are equivalent), 
and similarly for $\sigma^N_\beta$ and $\xr_\beta$. 
Note that the fact that $\bar{t}_\alpha$ and $\bar{t}_\beta$ may have different types is not important (as will be explained).
The third formula, $F$, will give the analogue of equation (\ref{eq:as}) ``points of intersection.''
After giving the definition, we will work towards proving the characterization in 
Theorem \ref{t:circle}.\footnote{In Definition \ref{d:circ}, note that $\ocirc$, ``the circle property'' which abstracts the above analysis 
will be the indicator of complexity ($\trg$ is $\xc$-unsuperstable), whereas 
its negation $\neg \ocirc$ will be the indicator of non-complexity ($\trg$ is $\xc$-superstable).} 

\begin{rmk}
\tcb{By an infinitary quantifier-free formula we will mean the disjunction, indexed by all quantifier-free types of tuples that 
satisfy the relation, of the conjunction of the formulas in the indexing type.} 
\end{rmk}

\begin{defn} \label{d:circ}
The context $\xc = (I, \mk)$ has property $\ocirc$ when:

\begin{quotation}
\noindent  For every finite $I_0 \subseteq I$ there is a finite $I_1$ with 
$I_0 \subseteq I_1 \subseteq I$, letting $\bar{s}$, $\bar{t}$ list  
$I_0$, $I_1$ respectively, such that for any $\aleph_0$-saturated $J \supseteq I$ there exist quantifier-free 
$($possibly infinitary$)$ formulas of $\tau(\mk)$ called $F(\bar{x}_1, \bar{x}_2; \bar{y})$,  
$E_1(\bar{x}_1, \bar{x}_2; \bar{y})$, 
$E_2(\bar{x}_1, \bar{x}_2; \bar{y})$, such that
 $\ell(\bar{x}_1) = \ell(\bar{x}_2) = \lgn(\bar{t})$, $\lgn(\bar{y}) = \lgn(\bar{s})$, and:  
\begin{itemize}
\item[(i)] for $\ii = 1, 2$ $E_\ii(\bar{x}_1, \bar{x}_2; \bar{s})$ defines an equivalence relation on 
$Y_{\bar{s}} = \{ \bar{t}^\prime \in {^{\lgn(\bar{t})}J} : 
\tpqf(\bar{t}^\prime, \bar{s}, J) = \tpqf(\bar{t}, \bar{s}, J) \}$.
\item[(ii)] $F(\bar{x}_1, \bar{x}_2; \bar{s})$  
defines a nonempty one-to-one partial function\footnote{\tcb{i.e., $F$ matches up certain $E_1$-classes (or their 
formal representatives) with certain $E_2$ classes.}}
 from $Y_{\bar{s}}/E_1(-,-; \bar{s})$ to $Y_{\bar{s}}/E_2(-,-; \bar{s})$, and 
\item[(iii)] $F$ has no fixed points, in other words for no $\bar{t} \in Y_{\bar{s}}$ is it the case that $F(\bar{t}, \bar{t}; \bar{s})$. 
\end{itemize}
\end{quotation}
\end{defn}

We now work towards a characterization, Theorem \ref{t:circle}.  
Note that the intent is $\ocirc_\xc$ means $\xc$ is too expressive; it expresses that the theory of the random graph is 
unsuperstable in some sense, whereas  $\neg \ocirc_\xc$ means $\xc$ is reasonable. 
For a high-level view of this property, see \S \ref{s:eq}.

\subsection{Discussion} \label{disc:x2} As a warm-up to Claim \ref{claim15a}, let us verify that indeed this property has captured 
$\xc$-shearing, by reversing the abstraction above.  Suppose that $\xc$ has property $\ocirc$. 
Fix a finite subset $I_0$ of $I_\xc$ and $\Psi \in \Upsilon[\trg]$. 
Let $I_1, \bar{s}, \bar{t}, E_1, E_2, F$ witness property $\ocirc$ for $I_0$. Fix an $\aleph_0$-saturated $J \in \mk_\xc$ with $I \subseteq J$, and 
we will find a formula of $\tau(\trg)$ which $(I_0, I_1)$-shears for $\xc$ as follows. 
Let $N = \GEM(J, \Psi)$ have skeleton $\langle \bar{a}_t : t \in J \rangle$.  Without loss of generality, 
$||N|| \geq |J|$. 
Let $\sigma_1, \sigma_2$ be two new 
$(\ell(\bar{a}_t)+\ell(\bar{a}_{\bar{s}}))$-place function symbols not already in $\tau(\Psi)$. 
Let $\xr = \tpqf(\bar{t}, I_0, I)$.  Let $Y_{\bar{s}} = \xr(J)$ be the set of realizations of $\xr$ in $J$. 

Expand $N$ to $N^+$ by interpreting $\sigma_1, \sigma_2$ as follows. 

\begin{itemize}

\item \emph{Formal description:}  We require the expansion to satisfy: 
for any $\bar{t}^\prime, \bar{t}^{\prime\prime} \in Y_{\bar{s}}$, 
\begin{enumerate}
\item[(a)] for $\ii = 1, 2$, 
\[  N^+ \models \sigma_{\ii}(\bar{a}_{\bar{t}^\prime}, \bar{a}_{\bar{s}}) = \sigma_{\ii}(\bar{a}_{\bar{t}^{\prime\prime}}, \bar{a}_{\bar{s}}). \]
if and only if $J \models E_{\ii}(\bar{t}^\prime, \bar{t}^{\prime\prime}; \bar{s})$. 

\item[(b)] $N^+ \models \sigma_{1}(\bar{a}_{\bar{t}^\prime}, \bar{a}_{\bar{s}}) = \sigma_{2}(\bar{a}_{\bar{t}^{\prime\prime}}, \bar{a}_{\bar{s}})$ 
if and only if $J \models F(\bar{t}^\prime, \bar{t}^{\prime\prime}; \bar{s})$. 
\end{enumerate}
Once we have done this, since $J$ is $\aleph_0$-homogeneous, we can also ensure that for 
any $\bar{u}$ realizing the same quantifier-free type as $\bar{s}$ in $J$ that the analogues of 
$(a), (b)$ hold with $\bar{u}$ in place of $\bar{s}$, remembering to then replace $\bar{t}$ by some $\bar{v}$ 
such that $\tpqf(\bar{v}^\smallfrown \bar{u}, \emptyset, J) = \tpqf(\bar{t} ~^\smallfrown \bar{s}, \emptyset, J)$. 

\br
\item \emph{Informal description}: Take any set of distinct elements of 
$\dom(N)$ of size $|J|$, thus $\geq |Y_{\bar{s}}| $, and interpret the functions to take values in this set according to the following 
informal heuristic. Given $\bar{t}^\prime \in Y_{\bar{s}}$, let us say ``the image of $\bar{t}^\prime$'' to mean 
$\bar{a}_{\bar{t}^\prime} ~^\smallfrown \bar{a}_{\bar{s}}$. Then: 
images of elements of $Y_{\bar{s}}$ are sent to the same $b \in \dom(N^+)$ by $\sigma_1$ if and only if they are in the same $E_1$-class; they are sent to the same $b \in \dom(N^+)$ by  $\sigma_2$ if and only if they are in the same $E_2$-class; and the values of $\sigma_1$ and 
$\sigma_2$ should coincide if and only if the $E_1$- and $E_2$-classes of the respective elements were matched by $F$. 
Ensure that the parallel conditions hold replacing $\bar{s}$ by any other $\bar{u}$ from $J$ with the same quantifier-free type. 

\end{itemize}
Let $\Phi^\prime \geq \Phi$ be given by applying the Ramsey property to $N^+$, so in $\Phi^\prime$ (and any template extending it) 
(a) and (b) will remain true, as will their analogues for $\bar{u} \equiv_{\operatorname{qf}} \bar{s}$. 

Now in the model $N^\prime = \GEM(I, \Phi^\prime)$, consider the formula
\[ R(x,\sigma_1(\bar{a}_{\bar{t}})) \land \neg R(x,\sigma_2(\bar{a}_{\bar{t}})). \]
Property (b) and the assumption \ref{d:circ}(iii) that $F$ has no fixed points ensure that 
in $N^\prime \models \sigma_1(\bar{a}_{\bar{t}} )\neq \sigma_2(\bar{a}_{\bar{t}})$, so this is a consistent formula. However, 
$F$ is a partial function and is nonempty, where non-emptiness is witnessed say by $N^\prime \models F(\bar{w}^\prime, \bar{w}^{\prime\prime}; \bar{s})$. 
Since $J$ is $\aleph_0$-homogeneous and $\bar{w}^\prime \in Y_{\bar{s}}$, 
for any other $\bar{t}^\prime \in Y_{\bar{s}}$ there is $\bar{t}^{\prime\prime} \in Y_{\bar{s}}$ such that 
\[ \tpqf(\bar{t}^\prime ~ ^\smallfrown \bar{t}^{\prime\prime}, \bar{s}, J) = \tpqf(\bar{w}^\prime ~ ^\smallfrown \bar{w}^{\prime\prime}, \bar{s}, J). \]
Since $F$ is an invariant of the quantifier-free type, this means $J \models F(\bar{t}^\prime, \bar{t}^{\prime\prime}; \bar{s})$. 
In short, the homogeneity of $J$ tells us that if $F$ is a partial one to one function it must be a bijection. 
It follows that
\[  \{ R(x,\sigma_1(\bar{a}_{\bar{t}}))  \land \neg R(x,\sigma_2(\bar{a}_{\bar{t}})) : \bar{t} \in Y_{\bar{s}} \} \]
is inconsistent in the following strong sense: for \emph{every} $\bar{t}_* \in Y_{\bar{s}}$, there is some $\bar{t}_{**} \in Y_{\bar{s}}$ 
such that 
\[ \{ R(x,\sigma_1(\bar{a}_{\bar{t}_*})) \land \neg R(x,\sigma_2(\bar{a}_{\bar{t}_{*}})), 
R(x,\sigma_1(\bar{a}_{\bar{t}_{**}}) )\land \neg R(x,\sigma_2(\bar{a}_{\bar{t}_{**}})) \} \]
is inconsistent. 
Finally, observe this clearly satisfies the definition of shearing, since the range of the map $f: Y_{\bar{s}} \rightarrow {^{2}N^\prime}$ given by 
\[ \bar{t}^\prime \mapsto ~\langle \sigma_1(\bar{a}_{\bar{t}^\prime}) \rangle ^\smallfrown  \langle \sigma_2(\bar{a}_{\bar{t}^{\prime}})\rangle \]
is a $\mk_\xc$-indiscernible sequence.  
This completes the discussion. 

\br Claim \ref{claim15a} now repeats this move in the context of an inductive 
argument, which gives the a priori stronger conclusion of $\xc$-unsuperstability; the minor but important new points to notice in the 
proof of \ref{claim15a} are the conditions there labelled (c), (e) which ensure that the 
images of the new Skolem functions at each inductive step are disjoint from those at earlier stages and 
allow us to ``continue along the independence property'' (in the random graph, disjointness is enough to ensure 
that the union of the formulas built at each step is indeed a type). 

\setcounter{theoremcounter}{3}

\begin{rmk} \label{rmk:ip} 
Given a complete theory $T$ and a formula $\vp(x,y)$, 
let $p(x)$ be a partial $\vp$-type in some $M \models T$. 
Let $\Gamma(x)$ be the infinite set of formulas expressing that $\vp(x,y)$ has the independence property. Observe that if 
$p(x) \cup \Gamma(x)$ is consistent, for any $\kappa$ there is $N \models T$, $M \preceq N$ containing a sequence 
$\langle b_i : i < \kappa \rangle$ such that $\vp$ has the independence property over this sequence [i.e., for any two finite 
disjoint $\sigma, \tau \subseteq \kappa$, $\{ \vp(x,b_i) : i \in \sigma \} \cup \{ \neg \vp(x,b_j) : j \in \tau \}$ is consistent] and 
moreover $p(x) \cup \{ \vp(x,b_i) \land \neg \vp(x,b_j) \}$ is consistent for any $i \neq j$. 
\end{rmk}

\begin{claim} \label{claim15a}
Let $\xc$ be a countable context and suppose $\xc$ has property $\ocirc$.  Then:
\begin{enumerate}
\item $\trg$ is $\xc$-unsuperstable; moreover, 
\item $(T, \vp)$ is $\xc$-unsuperstable, for any $T$, $\vp$ with the independence property.
\end{enumerate}
\end{claim}

\begin{proof} 
Let $\langle s_\ell : \ell < \omega \rangle$ list $I_\xc$. Define $I_n  = \{ s_\ell : \ell < n \}$, so each $I_n \in [I_\xc]^{<\aleph_0}$. 
By induction on $n$, we shall define an increasing sequence $\Phi_n \in \Upsilon[\trg]$ (so $n<m$ implies $\Phi_n \leq \Phi_m$) and 
an increasing sequence $p_n$ of partial $\vp$-types such that $p_n(x) \cup \Gamma(x)$ is consistent for each $n$, where $\Gamma$ is from 
\ref{rmk:ip}.  (In the case of the random graph, take $\vp(x,y) = R(x,y)$. 
The length of $x$ need not be 1, but we will drop overlines for simplicity.) 

For $n = 0$, let $M_0 = \GEM(I, \Phi_0)$ for some $\Phi_0 \in \Upsilon[\trg]$, and let $p_0 = \{ x = x \}$.  

For $n+1$, suppose we have defined $\Phi_n$ and $M_n = \GEM(I, \Phi_n)$. 
Let $\bar{s}_n$ list $I_{n,0} :=  I_n$.  Apply $\ocirc$ with $I_{n,0}$ in place of $I_0$. 
Let $I_{n,1}$, $F_{n}$, $E_{n,1}$, $E_{n,2}$ be as returned by the definition of $\ocirc$, and 
let $\bar{s}_n$, $\bar{t}_n$ list $I_{n,0}$ and $I_{n,1}$ respectively. 

Next we define $N_n$.  \tcb{Let $J_n \supseteq I_n$ be $\aleph_0$-saturated.}
In the case of the random graph, let $N_n = \GEM(J_n, \Phi_n)$. As $\Phi$ is nice, $M_n \preceq N_n$ and $N_n \rstr \tau(\trg) \models \trg$.  
For an arbitrary theory, let $N_n$ be an elementary extension of $\GEM(J_n, \Phi_n)$ such that $N_n \setminus M_n$ 
contains a sequence $\langle b_i : i < |J_n| \rangle$ 
as in \ref{rmk:ip} for $\vp$ and the partial type $p_n$. 
In this case it follows additionally that $M_n \preceq N_n$ and $N_n \rstr \tau(T) \models T$.

Let $\xr_n = \tpqf(\bar{t}_n, \bar{s}_n, J_n)$, so as usual ${\xr_n}(J_n)$ is the set of realizations of $\xr_n$ in $J_n$. 

We may expand $N_n$ to $N^+_n$ by adding two new functions $\sigma_{n,1}$ and $\sigma_{n,2}$, interpreted so that:

\begin{enumerate}
\item[(a)] for $\ii = 1, 2$, whenever $\bar{t}^\prime, \bar{t}^{\prime\prime} \in {\xr_n}(J_n)$, 
$N^+_n \models \sigma_{n,\ii}(\bar{a}_{\bar{t}^\prime}, \bar{a}_{\bar{s}_n}) = \sigma_{n,\ii}(\bar{a}_{\bar{t}^{\prime\prime}}, \bar{a}_{\bar{s}_n})$ 
if and only if $J_n \models E_{n,\ii}(\bar{t}^\prime, \bar{t}^{\prime\prime}; \bar{s}_n)$. 

\item[(b)] $N^+_n\models \sigma_{n,1}(\bar{a}_{\bar{t}^\prime}, \bar{a}_{\bar{s}_n}) = \sigma_{n,2}(\bar{a}_{\bar{t}^{\prime\prime}}, \bar{a}_{\bar{s}_n})$ 
if and only if $J_n \models F_{n}(\bar{t}^\prime, \bar{t}^{\prime\prime}; \bar{s}_n)$. 

\item[(c)] for $m < n$ and $\ii, \ell = 1, 2$,
$N^+_n\models \sigma_{n,\ii}(\bar{a}_{\bar{t}_n}, \bar{a}_{\bar{s}_n}) \neq \sigma_{m,\ell}(\bar{a}_{\bar{t}_m}, \bar{a}_{\bar{s}_m})$.

\item[(d)] For any $\bar{u}$ realizing the same quantifier-free type as $\bar{s}_n$ in $J$, the analogues of 
$(a), (b), (c)$ hold with $\bar{u}$ in place of $\bar{s}_n$ (remembering to then replace $\bar{t}_n$ by some $\bar{v}$ 
such that $\tpqf(\bar{v}^\smallfrown \bar{u}, \emptyset, J_n) = \tpqf({\bar{t}_n} ~^\smallfrown \bar{s}_n, \emptyset, J_n)$. 

\item[(e)] $\sigma_{n}(\bar{a}_{\bar{t}_n}, \bar{a}_{\bar{s}_n}) = b_i$ for some $i < \kappa$, i.e. the functions choose elements 
from our independent sequence. 
[For the random graph, essentially the whole model is an independent sequence, 
so it suffices to ask that $\sigma_{n}(\bar{a}_{\bar{t}}, \bar{a}_{\bar{s}_n}) \notin M_n$.]

\end{enumerate}

Let $\Phi_{n+1} \in \Upsilon[\trg]$ extending $\Phi_n$ be given by applying the Ramsey property (\cite{MiSh:1149} Corollary 2.10) 
to $N^+_n$, and then 
(a), (b), (c), (d) will remain true in any template extending $\Psi_{n+1}$.  Moreover, 
\[ p_{n+1} := p_n \cup \{ \vp(x,\sigma_{n,1}(\bar{a}_{\bar{t}_n}, \bar{a}_{\bar{s}_n})) \land \neg \vp(x, \sigma_{n,2}(\bar{a}_{\bar{t}_n}, \bar{a}_{\bar{s}_n})) \} \]
is consistent, and remains consistent with $\Gamma$ from \ref{rmk:ip}. 

Let $\Phi_\omega = \bigcup_n \Phi_n$, so $\Phi_\omega \in \Upsilon[T]$. Consider any $\Psi \geq \Phi_\omega$ and let 
$N = \GEM_{\tau(T)}(I_\xc, \Psi)$.    For each $n$, apply the non-fixed-point clause of \ref{d:circ}(iii) to
observe that 
\[ \neg F_n (\bar{t}_n, \bar{t}_n ; \bar{s}_n) \]
i.e. in the model $N$, 
\[ \sigma_{n,1}(\bar{a}_{\bar{t}_n}, \bar{a}_{\bar{s}_n}) \neq \sigma_{n,2}(\bar{a}_{\bar{t}_n}, \bar{a}_{\bar{s}_n}). \]
Remembering (c) and (e) above, this ensures the following set of formulas is a type:
\[ p(x) = \{ \vp(x,\sigma_{n,1}(\bar{a}_{\bar{t}_n}, \bar{a}_{\bar{s}_n})) \land \neg  \vp(x, \sigma_{n,2}(\bar{a}_{\bar{t}_n}, \bar{a}_{\bar{s}_n})) : n < \omega \}. \]
So $p$ is a partial type in $N \rstr \trg$, and let us show that for every $n$ it $\xc$-shears over $I_n$.  Why? 
First notice that $\bar{t}_n \subseteq I$ by construction. Second, notice that by the positive part of condition \ref{d:circ}(iii), 
for each $n$ there exists $\bar{u}_n  \in 
{\xr_n}(J_n)$ such that 
\[  N \models \sigma_{n,1}(\bar{a}_{\bar{t}_n}, \bar{a}_{\bar{s}_n}) = \sigma_{n,2}(\bar{a}_{\bar{u}_n}, \bar{a}_{\bar{s}_n}). \] 
This means that 
\tcb{
\begin{align*} 
\{ \vp(x,\sigma_{n,1}(\bar{a}_{\bar{t}^\prime_n}, \bar{a}_{\bar{s}_n})) \land & \neg  \vp(x, \sigma_{n,2}(\bar{a}_{\bar{w}^\prime_n}, \bar{a}_{\bar{s}_n})) :  \\
& \tpqf(\bar{t}^\prime_n ~^\smallfrown \bar{w}^\prime_n, \bar{s}_n, J_n) = \tpqf(\bar{t_n} ~^\smallfrown \bar{w}_n, \bar{s}_n, J_n) \} 
\end{align*}
}
will be inconsistent. 
\tcb{(Note that this may not a priori give $(I_n, I_{n+1})$-shearing for all $n$ as $\bar{t}_n$ may not be in $I_{n+1}$, but this will 
hold by reindexing.)} 

Thus $T$ is not $\xc$-superstable, as desired.  
\end{proof}

For the random graph, we also have a converse. 

\begin{claim} \label{c:west}
Let $\xc$ be a countable context and suppose $\trg$ is not $\xc$-superstable. Then $\xc$ has property $\ocirc$. 
\end{claim}

\begin{proof}
In this proof we appeal directly to the definition of $\xc$-shearing from 
Definition \ref{e5}, and 
$\xc$-unsuperstability, Definition \ref{m5a}. 
Let some finite $I_0 \subseteq I$ be given, and we will show how to find the rest of the data so that $\ocirc$ is satisfied. Recall 
from Definition \ref{m5a} that:
\begin{quotation}
we say $T$ is \emph{unsuperstable} for $\xc$ when there are: 
\begin{enumerate}
\item[(a)] 
an increasing sequence of nonempty finite sets $\langle I_n : n <\omega \rangle$ with 
$I_m \subseteq I_{n} \subseteq I$ for $m < n <\omega$ and $\bigcup_n I_n = I$, 
which are given along with a choice of enumeration $\bar{s}_n$ for each $I_n$ 
where $\bar{s}_n \tlf \bar{s}_{n+1}$ for each $n$ 
\item[(b)] 
an increasing sequence of nonempty, possibly infinite, sets $B_n \subseteq B_{n+1} \subseteq \mathfrak{C}_T$ in the monster 
model for $T$, with $B := \bigcup_n B_n$
\item[(c)] and a partial type  $p$ over $B$, such that
\end{enumerate}
\[ p \rstr B_{n+1}~~ (I_n, I_{n+1})\mbox{-shears over }B_n. \]
\end{quotation}
Let $I_n$  from Definition \ref{m5a} be a finite subset of $I$ containing $I_0$. 
Let $\bar{s}$ be an enumeration of $I_n$, and let $\bar{t}$ be an enumeration of $I_{n+1}$.
Let $N$ be the monster model for $\trg$. 
Let $p$ be the type given by \ref{m5a} which $(I_n, I_{n+1})$-shears.  [We can ignore the $B_{n+1}, B_n$ from that definition.] 
In the remainder of the proof, we use $I_0, I_1$ instead of $I_n, I_{n+1}$. 
Let $\vp(x,\bar{b}_{\bar{t}}) \in p$, $\mathbf{b}$ be the formula 
and $\mk$-indiscernible sequence in $N$ which witness this instance of shearing.  Recall from Definition \ref{e5} that:
\begin{quotation}
\noindent we say  
\emph{the formula $\vp(\bar{x}, \bar{c})$ shears }over $A$ in $M$ for $(I_0, I_1, \xc)$ 
when there exist a model $N$, a sequence $\mb$ in $N$, enumerations $\bar{s}$ of $I_0$ and $\bar{t}$ of $I_1$, and an $\aleph_0$-saturated $J \supseteq I$ such that:
\begin{enumerate}
\item $I_0 \subseteq I_1$ are finite subsets of $I$
\item $M \preceq N$
\item $\mb = \langle \bar{b}_{\bar{t}^\prime} : \bar{t}^\prime \in {^{\omega >}(J[I_0])} \rangle $ is $\mk$-indiscernible in $N$ over $A$
\item $\bar{c} = \bar{b}_{\bar{t}}$, and 
\item the set of formulas 
\[ \{ \vp(\bar{x}, \bar{b}_{\bar{t}^\prime}) : \bar{t}^\prime \in {^{\lgn(\bar{t})}(J)}, 
\tpqf(\bar{t}^\prime, \bar{s}, J) = \tpqf(\bar{t}, \bar{s}, I) \} \]
is contradictory.
\end{enumerate}
\end{quotation}
Note that by the conditions on 
$p$ in \ref{m5a}, the formula $\vp(\bar{x}; \bar{y}, \bar{z}, \bar{w})$ will be nonalgebraic.
\tcb{By quantifier elimination, it will be expressible as a disjunction of statements of the form: $x_i$ (some element of 
$\bar{x}$) has an edge to some of the $y$'s, a non-edge to some of the $z$'s (which are disjoint from the $y$'s), and is not equal to any of the $w$'s.  Since inconsistencies must come from asserting that some $x_i$ both 
connects and does not connect to the same parameter, we can reduce to considering a single disjunct of this form, 
and assuming $\lgn(\bar{x}) = 1$.} 

Let $J$ be $\aleph_0$-saturated such that $I \subseteq J \in \mk$.
Let $n = \lgn(\bar{b}_{\bar{t}})$. Fixing notation, let $\bar{b}_{\bar{t}} = \langle \bar{b}_{\bar{t}, i} : i < n \rangle$. 
Since $\vp(x,\bar{b}_{\bar{t}})$ is a consistent nonalgebraic formula in the random graph, we can write $n$ 
as the union of sets $A$, $B$, $C$ \tcb{(where $A, B$ are disjoint)} so that \tcb{without loss of generality}\footnote{$\vp$ is a disjunction of 
such formulas, so at least one is consistent.}
\[ \vp(x,\bar{b}_{\bar{t}}) \equiv \bigwedge_{i \in A} R(x,\bar{b}_{\bar{t}, i}) \land \bigwedge_{j \in B} \neg R(x,\bar{b}_{\bar{t}, j})
\land \bigwedge_{k \in C} x \neq \bar{b}_{\bar{t}, k}. \]
Define $Y_{\bar{s}} = \{ \bar{t}^\prime \in {^{\lgn(\bar{t})}J} : \tpqf(\bar{t}^\prime, \bar{s}, J) = \tpqf(\bar{t}, \bar{s}, J) \}$. 
In order to define $E_1, E_2$, observe that since the set of formulas 
\[ \{ \vp(\bar{x}, \bar{b}_{\bar{t}^\prime}) : \bar{t}^\prime \in {^{\lgn(\bar{t})}(J)}, 
\tpqf(\bar{t}^\prime, \bar{s}, J) = \tpqf(\bar{t}, \bar{s}, I) \} \]
is contradictory, there must be some $\bar{t}_\alpha, \bar{t}_\beta \in Y_{\bar{s}}$ and $i_\alpha \in A$, $j_\beta \in B$ such that 
\[ \bar{b}_{\bar{t}_\alpha, i_\alpha} = \bar{b}_{\bar{t}_\beta, j_\beta}. \]
Let $E_1(\bar{x}_1, \bar{x}_2; \bar{s})$ be the following two-place relation on $Y_{\bar{s}}$:
\[ E_1(\bar{t}_1, \bar{t}_2; \bar{s}) \mbox{ ~ iff ~} N \models \bar{b}_{\bar{t}_1, i_\alpha} = \bar{b}_{\bar{t}_2, i_\alpha} .  \]
Let $E_2(\bar{x}_1, \bar{x}_2; \bar{s})$ be the following two-place relation on $Y_{\bar{s}}$:
\[ E_2(\bar{t}_1, \bar{t}_2; \bar{s}) \mbox{ ~ iff ~} N \models \bar{b}_{\bar{t}_1, j_\beta} = \bar{b}_{\bar{t}_2, j_\beta}.  \]
Then for $\ell = 1, 2$, clearly:
\begin{itemize}
\item $E_\ell = E_\ell(\bar{x}_1, \bar{x}_2; \bar{s})$ is an equivalence relation on $Y_{\bar{s}}$. 
\item The truth value of 
``$E_\ell(\bar{t}_1, \bar{t}_2; \bar{s})$'' is determined by $\tpqf({\bar{t}_1} ^\smallfrown \bar{t}_2, \bar{s}, J)$, for any  
$\bar{t}_1, \bar{t}_2 \in Y_{\bar{s}}$, in other words, it is an invariant of the quantifier-free type (thus, said to be 
definable by a possibly infinitary quantifier-free formula).
\end{itemize}
Finally, define $F(\bar{x}_1, \bar{x}_2; \bar{s})$, a two place relation on $Y_{\bar{s}}$, by: 
$F(\bar{t}_1, \bar{t}_2; \bar{s})$ if and only if 
\[ N \models \bar{b}_{\bar{t}_1, i_\alpha} = \bar{b}_{\bar{t}_2, j_\beta} .  \]
Then $F(\bar{x}_1, \bar{x}_2; \bar{s})$  
defines a nonempty one to one partial function from $Y_{\bar{s}}/E_1(-,-; \bar{s})$ to $Y_{\bar{s}}/E_2(-,-; \bar{s})$.

It remains to check that $F$ has no fixed points.  Recall that for our original tuple $\bar{t}$, 
we know that 
\[\{ \bar{b}_{\bar{t}, i} : i \in A \} \cap \{ \bar{b}_{\bar{t}, j} : j \in B \} = \emptyset \] 
because $\vp(x,\bar{b}_{\bar{t}})$ is a consistent $\tau(\trg)$-formula in $N$. Thus, 
$J \models \neg F(\bar{t}, \bar{t}; \bar{s})$. Since $F$ is definable by a possibly 
infinitary quantifier-free formula,  this remains true for all tuples from $Y_{\bar{s}}$.  
This completes the verification of $\ocirc$. 
\end{proof}

\subsection{Discussion: obtaining $\ocirc$ directly from the example in \ref{sub:ex}} For completeness, we 
work out the analysis of \ref{sub:ex}.  
Recall equations (\ref{eq:as1}) and (\ref{eq:as}) there.   
Recall that $\bar{t}_\alpha, \bar{t}_\beta$ come from equation (\ref{eq:p}) for our given $\alpha$, $\beta$ 
(so $\bar{v}_*$ and $\bar{w}_*$ share their respective quantifer-free types over $\bar{s}$).  
As $J$ is $\aleph_0$-homogeneous, there are 
$\bar{v}$, $\bar{w}$ from $J$ such that 
\[ \tpqf(\bar{v}_* ~^\smallfrown \bar{w}, \bar{s}, J) = \tpqf(\bar{v} ~^\smallfrown \bar{w}_*, \bar{s}, J)  
= \tpqf(\bar{t}_\alpha ~^\smallfrown \bar{t}_\beta, \bar{s}, I).     \]
(It may be that $\bar{v} = \bar{v}_*$ or $\bar{w} = \bar{w}_*$, and if so, no problem.)

Let $\bar{t} = \bar{s} ~^\smallfrown \bar{t}_\alpha ~^\smallfrown \bar{t}_\beta$.
Without loss of generality, we may simultaneously replace 
$\bar{v}_*$ and $\bar{w}_*$ respectively by 
\[  \bar{s} ~^\smallfrown \bar{v}_* ~^\smallfrown \bar{w}, 
 \mbox{ \hspace{3mm} } \bar{s} ~^\smallfrown \bar{v} ~^\smallfrown \bar{w}_*. \] 
So for us, 
\[ Y_{\bar{s}} =  \{ \bar{t}^\prime \in {^{\lgn(\bar{t})}J} : 
\tpqf(\bar{t}^\prime, \bar{s}, J) = \tpqf(\bar{t}, \bar{s}, J) \} \]
and this set includes $\bar{v}_*$ and $\bar{w}_*$. 
After possibly adding dummy variables to $\sigma_\alpha$, $\sigma_\beta$ we may assume their operation is unchanged. 
Let $E_1(\bar{x}_1, \bar{x}_2; \bar{s})$ be the following two-place relation on $Y_{\bar{s}}$:
\[ \mbox{ $E_1(\bar{t}_1, \bar{t}_2; \bar{s})$  ~ iff ~ $N \models$ ``$\sigma_\alpha(\bar{a}_{\bar{t}_1}) = \sigma_\alpha(\bar{a}_{\bar{t}_2})$''. } \]
Let $E_2(\bar{x}_1, \bar{x}_2; \bar{s})$ be the following two-place relation on $Y_{\bar{s}}$:
\[ \mbox{ $E_2(\bar{t}_1, \bar{t}_2; \bar{s})$  ~ iff ~ $N \models$ ``$\sigma_\beta(\bar{a}_{\bar{t}_1}) = \sigma_\beta(\bar{a}_{\bar{t}_2})$''. } \]
Then for $\ell = 1, 2$, clearly:
\begin{itemize}
\item $E_\ell = E_\ell(\bar{x}_1, \bar{x}_2; \bar{s})$ is an equivalence relation on $Y_{\bar{s}}$. 
\item The truth value of 
``$E_\ell(\bar{t}_1, \bar{t}_2; \bar{s})$'' is determined by $\tpqf({\bar{t}_1} ^\smallfrown \bar{t}_2, \bar{s}, J)$, for any  
$\bar{t}_1, \bar{t}_2 \in Y_{\bar{s}}$, in other words, it is an invariant of the quantifier-free type (thus, said to be 
definable by a possibly infinitary quantifier-free formula).
\end{itemize}
Finally, define $F(\bar{x}_, \bar{x}_2; \bar{s})$, a two place relation on $Y_{\bar{s}}$, by: 
$F(\bar{t}_1, \bar{t}_2; \bar{s})$ if and only if 
\[ N \models \sigma_\alpha(\bar{a}_{\bar{t}_1}) = \sigma_\beta(\bar{a}_{\bar{t}_2}).  \]
Then $F$ naturally defines a subset $X_{\bar{s}}$ of $Y_{\bar{s}} \times Y_{\bar{s}}$ i.e. if 
$\bar{t}_1, \bar{t}_2 \in Y_{\bar{s}}$ then the truth value of $F(\bar{t}_1, \bar{t}_2; \bar{s})$ is the same for any 
$(\bar{t}^\prime_1, \bar{t}^\prime_2) \in (\bar{t}_1/E_1(-,-,\bar{s})) \times (\bar{t}_2/E_2(-,-,\bar{s}))$. 
$F$ is not empty because of equation (\ref{eq:as}), and respects the equivalence relations.  
(If $(\bar{t}^\prime_\ell, \bar{t}^{\prime\prime}_\ell) \in F$ for $\ell = 1, 2$ then 
$\bar{t}^\prime E_1 \bar{t}^\prime_2$ iff $\bar{t}^{\prime\prime}_1 E_2 \bar{t}^{\prime\prime}_2$. 
Together we get that $F$ is a 1-to-1 partial function from $Y/E_1$ into $Y/E_2$, so by homogeneity of $J$ it is a 
function \tcb{with full domain and range}.)

It remains to show that $F$ has no fixed points. Now, whether or not $F(\bar{u}, \bar{u}; \bar{s})$ is an invariant of the quantifier-free type 
$\qftp(\bar{u}, \bar{s}, J)$.  We know that 
\[ N \models \sigma_\alpha(\bar{a}_{\bar{t}}) \neq \sigma_\beta(\bar{a}_{\bar{t}})  \]
by our definition of $\bar{t}$, because the original type $p$ was consistent. It follows that 
\[ \neg F(\bar{t}, \bar{t}; \bar{s}) \]
 and that this is an invariant of $\qftp(\bar{t}, \bar{s}, I)$. 
So it will remain be true for any tuple from $Y_{\bar{s}}$.
This proves $F$ has no fixed points, which completes the verification of $\ocirc$ and the discussion.

\vspace{3mm}

Summarizing, we arrive at:

\begin{theorem} \label{t:circle} 
Let $\xc$ be a countable context. 
$\trg$ is $\xc$-unsuperstable if and only if $\xc$ has property $\ocirc$.
\end{theorem}

\begin{proof}
Claim \ref{claim15a} and Claim \ref{c:west}. 
\end{proof}

We now give a positive and a negative example of $\ocirc$. 
First, we verify that $\trg$ is $\xc$-unsuperstable for contexts coming from linear orders. 
In the following example, the context need not be countable. 

\begin{expl} \label{xyz-lin}
Let $\mk$ be the class of infinite linear orders, Let $I \in \mk$ \tcb{be $\aleph_0$-saturated}. Then $\xc = (I, \mk)$ has $\ocirc$.
\end{expl}

\begin{proof}
Let $M \models \trg$ and let $\langle a_t :  t \in I \rangle$ be any infinite indiscernible sequence. 
Recall that if 
$\bar{t} = t_0t_1$ we let $\bar{a}_{\bar{t}}$ denote ${a_{t_0}}^\smallfrown a_{t_{1}}$.  
Choose $t^*_0, t^*_1 \in I$ with $t^*_0 < t^*_1$. Let $\xr = \tpqf({t^*_0} ^\smallfrown {t^*_1}, \emptyset, I)$.  
(In this example, $\bar{s} = \langle \rangle$, $\bar{t} = \langle t^*_0, t^*_1 \rangle$ are the corresponding values from $\ocirc$.)
Recall that ${\xr}(I)$ is the set of realizations of 
$\xr$ in $I$.  Then 
\[ \{ \bar{a}_{\bar{t}} : \bar{t} \in {^2 I},  \bar{t} \in {\xr}(I) \} \]
is a $\mk$-indiscernible sequence. 
Define $E_0, E_1, E_2$ on $\xr(I)$ as follows: 
\[ E_1(t_1t_2, t^\prime_1 t^\prime_2) \equiv t_1 = t^\prime_1 \]
Now 
\[ E_2(t_1t_2, t^\prime_1 t^\prime_2) \equiv t_2 = t^\prime_2 \]
\[ E_0(t_1t_2, t^\prime_1 t^\prime_2) \equiv t_1 = t^\prime_2. \]
We can also see the shearing directly: \tcb{since $\bar{a}_{\bar{t}}$ for $\bar{t} \in {\xr}(I)$ has length two, let 
$\vp = \vp(x;y,z)$ and then the set}
\[   \{ \vp(x,\bar{a}_{\bar{t}}) : \bar{t} \in \xr(I) \}  \]
is inconsistent since for every $t_1t_2 \in \xr(I)$ \tcb{and the homogeneity of $I$}, 
there is $t_0 \in I$ such that $t_0 t_1 \in \xr(I)$. 
\end{proof}

There are also natural examples which do not have $\ocirc$. In the following example, 
$I$ also need not be countable. 

\begin{expl}  \label{xyz-nonlin}
Let $\mk_\mu$ be the class of linear orders expanded by $\mu$ unary predicates which partition the domain. 
Let $I \in \mk_\mu$ and suppose that for each predicate $P_\alpha$, $|P^I_\alpha | \leq 1$. Let 
$\xc = (I, \mk_\mu)$. Then $\xc$ does not have $\ocirc$. 
\end{expl}

\begin{proof}
See \cite{MiSh:1124} Claim 5.10, where it is shown that for any such $\xc$ and $M = \GEM(I, \Phi)$, 
if $p$ is a type of $M \rstr \tau(\trg)$ then there is $\Psi \geq \Phi$ in which $p$ is realized. 
(It follows that $\trg$ is $\xc$-superstable.)  \emph{For the interested reader, the main idea of that proof is 
quite close to that of Claim \ref{trg-ss} below; just leave out the relations $R$ on $J$.}
\end{proof}

\br

\begin{concl} \label{c:unstable}
The following are equivalent:  
\begin{enumerate}
\item[(a)]  The formula $\vp$ is stable $($with respect to the complete theory $T$). 

\item[(b)]  $(T, \vp)$ is $\xc$-superstable for every countable context $\xc$. 

\end{enumerate}
\end{concl}

\begin{proof}
For (a) implies (b): This is because stable formulas have definitions, thus a fortiori weak definitions. We assemble 
some facts from \cite{MiSh:1149} for a more detailed proof. 
Suppose for a contradiction that $(T, \vp)$ were $\xc$-unsuperstable for some countable context $\xc = (I, \mk)$.  
Let $M = \GEM(I, \Phi)$. 
Then \cite{MiSh:1149} Corollary 6.16 would build a larger template $\Phi_*$ 
so that for any $\Psi \geq \Phi_*$, the model $\GEM(I, \Psi)$ is not $\aleph_1$-saturated for $\vp$-types.  
In particular, that construction builds a $\vp$-type $p$ which does not have a weak $\bar{t}_*$-definition for any finite 
$\bar{t}_*$ in $I$, and moreover cannot have one in any $\GEM(I, \Psi)$ for any $\Psi \geq \Phi_*$, see \cite{MiSh:1149} 
Remark 6.15 or Corollary 4.25.  However, 
by \cite{MiSh:1149} Claim 4.15, in any $\GEM(I, \Psi)$ any type in a stable formula has a definition, thus a weak definition. 
Contradiction.

For (b) implies (a): Suppose first that $\vp$ is not simple (has the tree property). Then it has long dividing chains, and since 
dividing implies shearing (\cite{MiSh:1149}, Claim 5.8 and Remark 5.9), $\vp$ is not $\xc$-superstable. 
If $\vp$ is simple unstable, $\vp$ has the independence property, and we can apply Claim \ref{claim15a}(b). 
\end{proof}

\br
\section{Interlude on ``$\operatorname{eq}$''}  \label{s:eq}
\setcounter{theoremcounter}{0}

In this section, we observe that it isn't an accident that our negative
examples of $\ocirc$ have a certain form. 
There is a nice further explanation of this, and of $\ocirc$, once we define the analogue of $M^{\operatorname{eq}}$ in this context. 

\begin{defn} 
We say that the countable context $\xc = (I, \mk)$ is \emph{essentially separated} when there is a finite $I_0 \subseteq I$ 
such that $s \neq t \in I$ implies that $\tpqf(s, I_0, I) \neq \tpqf(t, I_0, I)$. 
\end{defn}

\begin{conv}
\tcb{In this section, take $\xc$ to be arbitrary but fixed, and write $\mk$, $I$ for $\mk_\xc$, $I_\xc$ respectively.}  
 \end{conv}
 
The idea is that  ``$\neg \ocirc$'' is for all intents and purposes 
essentially separated (one might say, is essentially \emph{essentially separated}). 
In this section, we outline a proof of this, which involves defining the analogue of $M^{\operatorname{eq}}$ for contexts, 
defining for any context $\xc$ a so-called ``$\operatorname{eq}$-extension'' 
and then pointing out that 
the property $\neg \ocirc$ for $\xc$ is really saying that some such ``$\xc^{\eq}$'' is well behaved. 
Since the section is primarily explanatory, we will be brief. 

\begin{defn} \label{a2}
We say the context $\xd$ is an $\eq$-extension of our context $\xc$ when there is 
$\bar{E}$ such that:
\begin{enumerate}
\item $\bar{E} = \langle E_i : i < i_* \rangle$, $\bar{\vp} = \langle \vp_j : j < j_* \rangle$
\item $E_i = E_i(\bar{x}_i, \bar{y}_i)$ is a possibly infinitary quantifier-free formula in the vocabulary $\tau(\mk)$, with 
$\lgn(\bar{x}_i) = \lgn(\bar{y}_i) =:n_i$, such that for every 
$I^\prime \in \mk$ it defines an equivalence relation on tuples of $I^\prime$ of length $n_i$.  \\ We stipulate $E_0 = $``$x=y$''.  
$($Alternately, list the defining formulas $\vp_i$ for $E_i$ separately.$)$
\item For $I \in \mk$, let $I^+$ be the analogue of $\mathfrak{C}^{\eq}$ using the $E_i$'s for $i<i_*$.\footnote{
\tcb{Why not all possible  $E$s? It seems better to build the restriction into the definition, as otherwise 
for some values of $i_*$, $j_*$ we may lose countability.}} \\ That is, 
\begin{enumerate}
\item the universe of $I^+$ is $I \cup \{ \bar{t}/E^{I}_i : i < i_*, \bar{t} \in {^{n_i}I } \}$,  
\\ but we identify $t/E_0$ with $t$ for $t \in I$.
\end{enumerate}
\tcb{and as a signature we have the following symbols, given with their interpretations:}
\begin{enumerate}
\item $P^{I^+} = P^I$ for $P \in \tau(\mk)$ a predicate. 

\item $P^{I^+}_* = I$. 

\item for $F \in \tau(\mk)$ a function symbol, if any, $F^{I^+} = F$, so its domain is defined.

\item $F^{I^+}_i = \{ (\bar{t}, \bar{t}/E^{I}_i) : \bar{t} \in {^{n_i}I}   \}$.

\item $P^{I^+}_{\vp_j} = \{ \langle \bar{t}_{i_0}/E_{i_0}, \dots, \bar{t}_{i_{m(j)-1}}/E_{i_{m(j)-1}} \rangle : 
m(j) < \omega, i_\ell < i_*$, and $I \models \vp[\bar{t}_{i_0}, \bar{t}_{i_1}, \dots ] \}$. 

\item $\mk_\xd = \{ I^+ : I \in \mk_\xc \}$, and $I_\xd = (I_\xc)^+$.

\item \tcb{Summarizing, letting $\tau = \tau(\mk_\xd)$ consist of the symbols just given, } we have defined $\xd = (I_\xd, \mk_\xd)$. 

\end{enumerate}
\end{enumerate}
\end{defn}

\begin{defn}
We say $\xd = \xc^{\eq}$ when we use all possible $E_i, \vp_i$ up to equivalence. $($Since we have required these 
sequences to be countable, this of course puts some restrictions on the contexts $\xc$ for which $\xc^{\eq}$ is presently defined. We will 
not at present require such a canonical extension to exist, but it is reasonable to define it.$)$
\end{defn}

\begin{claim} \label{a5}
For any countable context $\xc$, any $\xd$ defined from it as in $\ref{a2}$ is also a context.   Moreover, if 
$\lambda \geq |I_\xc| + | \{ \tp(\bar{s}, \emptyset, I) : \bar{s} \in {^{\omega >}I}, I \in \mk_\xc \}|$ then also 
$\lambda \geq |I_\xd| + | \{ \tp(\bar{s}, \emptyset, I) : \bar{s} \in {^{\omega >}I}, I \in \mk_\xd \}|$. 
\end{claim}

\begin{defn} \label{a10}
For $I \in \mk_\xc$ and finite $\bar{s} \in {^{\omega >}I}$: 
\begin{enumerate}
\item let $\dcl(\bar{s}, I) = \{ t \in I : $ we cannot find $J \in \mk_\xc$, $\bar{w} \in {^{\lgn{\bar{s}}}J}$ and $t_1 \neq t_2 \in J$ 
such that for $\ell = 1, 2$, 
$\tpqf({\bar{w}}^\smallfrown \langle t_\ell \rangle, \emptyset, J) = \tpqf({\bar{s}}^\smallfrown \langle t \rangle, \emptyset, J)$. 

\item let $\acl(\bar{s}, I)$ be defined similarly, replacing ``$t_1 \neq t_2$'' by $\langle t_\ell : \ell < \omega \rangle$. 
\end{enumerate}
\end{defn}

\begin{claim} \label{a13}  For a countable context $\xc$ the following are equivalent. 

\begin{enumerate}
\item[(a)] $\xc$ satisfies $\ocirc$. 

\item[(b)] for some $\eq$-extension $\xd$ of $\xc$, for every finite $\bar{s}$ from $I_\xd$, there are $r_0 \neq r_1 \in I_\xd$ realizing the same 
complete quantifier-free type over $\bar{s}$.

\end{enumerate}
\end{claim}

\begin{proof} 
\tcb{Informally, $\ocirc$ speaks about particular equivalence relations being in an anti-diagonal correspondence, 
and $\eq$-extensions speak about all equivalence relations; so this is observing that we should be able to use existence of the first 
to get elements which do not look alike to the second, and vice versa. 
The slightly longer direction is (b) implies (a).  Let $\xd$ be the eq-extension of $\xc$ given by hypothesis, and notice that as underlying sets 
$I_\xc = I_\xd$, though of course as models the second is an expansion of the first. Fix 
some finite $I_0 \subseteq I_\xc$ (so also $\subseteq I_\xd$). Let $\bar{s}$ enumerate $I_0$, and we have to find 
$I_1$, $\bar{t}$, $E_1$, $E_2$ as promised. 
By hypothesis we can find $r_0 \neq r_1 \in I_\xd$ (so also in the set $I_\xc$) realizing the same complete quantifier-free type over $\bar{s}$. 
Let $\bar{t}_1$ and $E_1$ be such that $r_1 = \bar{s}^\smallfrown{\bar{t}_1}/E_1$, and $\bar{t}_1 \subseteq I_\xc$.  
Let $\bar{t}_2$ and $E_2$ be such that $r_2 = \bar{s}^\smallfrown{\bar{t}_2}/E_2$, and $\bar{t}_2 \subseteq I_\xc$. 
Without loss of generality, $\bar{t}_1 = \bar{t}_2$, call it $\bar{t}$, and without loss of generality 
$\bar{s}^\smallfrown \bar{t}$ has no repetitions. Let $J \in \mck_\xc$ be $\aleph_0$-saturated extending $I_\xc$, so also its corresponding 
expansion $J^+ \in \mck_\xd$ is $\aleph_0$-saturated extending $I_\xd$.  
We define $E_i(\bar{x}_1, \bar{x}_2, \bar{y})$ where $\lgn(\bar{x}_1) = \lgn(\bar{x}_2) = \lgn(\bar{t})$, $\lgn(\bar{y}) = \lgn(\bar{s})$ by:
$E_i(\bar{t}^\prime_1, \bar{t}^\prime_2, \bar{s}^\prime)$ iff [$\bar{s}^\prime ~^\smallfrown \bar{t}^\prime_1$ realizes 
$\tpqf(\bar{s}^\smallfrown \bar{t}, \emptyset, I_\xc)$, and 
$\bar{s}^\prime ~^\smallfrown \bar{t}^\prime_2$ realizes 
$\tpqf(\bar{s}^\smallfrown \bar{t}, \emptyset, I_\xc)$, and 
if $r, r^\prime$ are in $I_\xd$ and $\bar{s}^\prime ~^\smallfrown \bar{t}^\prime ~^\smallfrown r^\prime$, 
$\bar{s}^\prime ~^\smallfrown \bar{t}^\prime ~^\smallfrown \bar{r}^{\prime\prime}$ realize $\tpqf(\bar{s}^\smallfrown \bar{t} ~^\smallfrown r_\ell, 
\emptyset, I_\xc)$]. 
It remains to define $F$, as follows. In $J^+$ we let 
\[ F = \{ (\bar{u}^\smallfrown \bar{v}/E_1, \bar{u}^\smallfrown \bar{v}/E_2) : \bar{u}^\smallfrown \bar{v} \subseteq J^+ \mbox{ realizes } 
\tpqf(\bar{s}^\smallfrown\bar{t}, \emptyset, I_\xd = I^+_\xc )\}. \] 
For (a) implies (b), let $r_\ell = \bar{t}_\ell/E^\ell_{\bar{s}}$ for $\ell = 1, 2$ as in the definition of $\ocirc$. 
Then $r_1 = \bar{t}/E^1_{\bar{s}}$, $r_2 = \bar{t}/E^2_{\bar{s}}$, and $F_{\bar{s}}(r_1) = r_2$, hence $r_1 \neq r_2$. }
\end{proof}

\vspace{3mm}

\section{Separating the random graph and $T_{n,k}$}
\setcounter{theoremcounter}{0}

We now show one can separate $\trg$ and $T_{3,2}$ using countable contexts. 
That is, for each $n>k\geq 2$, 
we prove that there are countable contexts 
$\xc$ for which $\trg$ is $\xc$-superstable but $T_{n,k}$ is  $\xc$-unsuperstable, Theorem \ref{t:disting}.

\begin{defn} \label{d:k-nk}
Let $\mk_{n,k}$ be the following index model class. In $\tau(\mk_{n,k})$ we have 
a binary relation $<$, unary predicates $\{ P_q : q \in \mathbb{Q} \}$, and a $(k+1)$-place relation $R$, and on $I \in \mk$:
\begin{itemize}
\item $<$ is a linear order
\item the $P_q$ are disjoint unary predicates which partition $I$
\item $R$ is symmetric irreflexive (a hyperedge), and has no cliques of size $n+1$, i.e. 
for any distinct $i_0, \dots, i_{n}$ from $I$ it is not the case that $R^I$ holds on all $k+1$-element subsets of 
$\{ i_0, \dots, i_n \}$. [Note this is a universal statement.]
\end{itemize}
\end{defn}

\begin{claim}
$\mk_{n,k}$ is a Ramsey class and $($satisfies our hypotheses$)$. 
\end{claim}

\begin{proof}
The key point is being a Ramsey class. By a theorem of Scow, 
it suffices to check this for the class of finite substructures of members of $\mk$, see \cite{scow3} Theorem 3.12, 
 so we may cite the general Ne\v{s}et\v{r}il-R\"{o}dl theorem for relational structures, see \cite{hn} Theorem 3.21.
\end{proof}

For the next few claims we shall use:

\begin{defn} \label{d:xcnk} Given $n>k\geq 2$, 
let $\xc_{n,k}$ be the context $(I, \mk)$ where $\mk  = \mk_{n,k}$ and 
$I$ has domain $\mathbb{Q}$, $<^I$ is the usual linear order on $\mathbb{Q}$, $P^I_q = \{ q \}$, and 
$R^I = \emptyset$. 
\end{defn}

\begin{disc} \emph{Clearly each $\xc_{n,k}$ is a countable context. Recall the intent of $n,k$: the edge relation $R$ has 
arity $k+1$, and the forbidden configuration is a `large clique' where large means size $n+1$. 
Note that since we make no demands on $R$ in \ref{d:k-nk} other that having no large cliques, there is no problem 
choosing an $I$ in which are there no instances of $R$ at all.  However, the point will be that when we consider 
$\aleph_0$-saturated $J \supseteq I$, instances of $R$ will appear.}
\end{disc}

\begin{claim} \label{trg-ss}
$\trg$ is $\xc_{n,k}$-superstable, for any $n>k\geq 2$.
\end{claim}

\begin{proof}
Fix $n,k$ for the course of the proof and let $\xc = \xc_{n,k}$, $I = I_\xc, \mk = \mk_\xc$.
It will suffice to show there is no instance of shearing for non-algebraic formulas $\vp(x,\bar{a})$ in the monster model for the random graph. 

\br
\noindent\emph{Step 1: Analysis of the problem.}
Suppose for a contradiction that there were some:
\begin{itemize}
\item finite $\bar{s}, \bar{t}$ from $I$ 
({without loss of generality 
$\bar{t}$ is strictly increasing 
without repetition, and $\bar{s}$ is a subsequence of $\bar{t}$})
\item $\xr = \qftp(\bar{t}, \bar{s}, I)$
\item $J$ which is 
$\aleph_0$-saturated with $I \subseteq J \in \mk$, 
\item a $\mk$-indiscernible sequence $\{ \bar{b}_{\bar{t}^\prime} : \bar{t}^\prime \in \xr(J) \}$ in the 
monster model for $\trg$,   
\item and a formula $\vp(x,\bar{y})$ with $\ell(\bar{y}) = \ell(\bar{b}_{\bar{t}})$, 
\end{itemize}
such that 
\begin{equation}
\label{eq:cont0}
\vp(x,\bar{b}_{\bar{t}}) \mbox{ is consistent  and non-algebraic }
\end{equation}
however 
\begin{equation}
\label{eq:cont1} \{ \vp(x,\bar{b}_{\bar{t}^\prime}) : \bar{t}^\prime \in \xr(J) \}  \mbox{ is inconsistent.} 
\end{equation}
As before, it suffices to consider the case where $\ell(x) = 1$, and we may partition 
$m = \ell(\bar{b}_{\bar{t}})$ into sets $A$, $B$, $C$ \tcb{(with $A$, $B$ disjoint, and by nonalgebraicity, $C = m$)} such that 
\[ \vp(x,\bar{b}_{\bar{t}}) \equiv \bigwedge_{i \in A} R(x,\bar{b}_{\bar{t}, i}) \land \bigwedge_{j \in B} \neg R(x,\bar{b}_{\bar{t}, j})
\land \bigwedge_{k \in C} x \neq \bar{b}_{\bar{t}, k}. \]
As we have excluded dividing because of equality, 
the contradiction in (\ref{eq:cont1}) must be because we have $\bar{v}, \bar{w} \in \xr(J)$ and 
$i, j$ such that $i \in A$, $j \in B$, and 
\begin{equation} \label{eq:coll} 
\bar{a}_{\bar{v}, i} = \bar{a}_{\bar{w}, j}. 
\end{equation}
Recall here that $\bar{a}_{\bar{v}, i}$ denotes the $i$-th element of the tuple $\bar{a}_{\bar{v}}$. 

\br
\noindent\emph{Step 2: A property that we would like $\bar{v}, \bar{w}$, our witnesses to collision, to have.} 
Let $\bar{u}$ enumerate, in increasing order, the set 
$\rn(\bar{v}) \cap \rn(\bar{w})$ of elements common to both sequences. 
(By the second line of Step 1, this will include the elements of $\bar{s}$.) 
Here we will follow an idea from the proof of 
\cite{MiSh:1124}, Claim 5.10. By ``an interval of consecutive elements of $\bar{u}$'' we shall mean a set of elements which are all 
less than $u_0$, or all greater than $u_{\lgn(\bar{u})-1}$, or all strictly between $u_i, u_{i+1}$ for some 
$0 \leq i < \lgn(\bar{u})-1$. 

Now consider the following potential property of $\bar{v}, \bar{w}$: 
\begin{quotation}
$(\star)$ within each interval of consecutive elements of $\bar{u}$, 
all elements of $\bar{v}$ falling in this interval are strictly below all elements of $\bar{w}$ falling in the same interval. 
\end{quotation}
In this step and the next, let us show that we may assume our $\bar{v}, \bar{w}$ (which were chosen to witness the contradiction) also 
satisfy $(\star)$. 

Suppose not, that is, suppose we chose our $\bar{v}$, $\bar{w}$ so that\footnote{Informally, the sum over all intervals of $\bar{u}$  
of the number of elements of $\bar{w}$ less than elements of $\bar{v}$ in each given interval.} 
\[ \{ (i,j) : v_i, w_j \mbox{ fall in the same interval of $\bar{u}$ but \tcb{$w_j \leq v_i$} } \} \]
is minimized, but 
we were not able to choose this number to be zero. That is, within at least one interval, say
$(u_i, u_{i+1})$ of $\bar{u}$, we have 
elements $v_k, w_j$ such that the following holds. (If one of the endpoints is 
$+\infty$ or $-\infty$, the same argument applies substituting this notation throughout.) 
\begin{equation*}
\begin{split}
u_i < \{ v \in \bar{v} : ~u_i < v < v_k \} \cup  \{ & w \in \bar{w} :  ~u_i < w < w_j  \} \\ 
& < w_j ~\mbox{\tcb{$\leq$}}~ v_k <  \\ 
 \{ v \in \bar{v} :  v_k < & ~ v < u_{i+1} \} \cup  \{ w \in \bar{w} : ~ w_j < w < u_{i+1} \} < u_{i+1}  \\
\end{split}
\end{equation*}
where some or all of the sets in the first and third lines may be empty. 
Recalling that $J$ is $\aleph_0$-saturated, we will justify in the next step that we may choose $w^\prime_j, v^\prime_k$
so that  
\begin{itemize}
\item  $w_j < v^\prime_k < w^\prime_j < v_k$ ~\emph{and}
\item writing $\bar{w}^\prime$ for the result of substituting $w^\prime_j$ for $w_j$ in $\bar{w}$, and 
writing $\bar{v}^\prime$ for the result of substituting $v^\prime_k$ for $v_k$ in $\bar{v}$, we have that 
\[ \tpqf(\bar{v}^\smallfrown \bar{w}, \bar{s}, J) = \tpqf(\bar{v}^\prime ~^\smallfrown \bar{w}, \bar{s}, J) = 
\tpqf(\bar{v}^\smallfrown \bar{w}^\prime, \bar{s}, J). \]
\end{itemize}
Why is this sufficient? Recalling that (\ref{eq:coll}) is an invariant of $\tpqf(\bar{v}^\smallfrown \bar{w}, \bar{s}, J)$, we will then have that 
\begin{align*}
\bar{a}_{\bar{v}, i} = & \bar{a}_{\bar{w}, j} \\
\bar{a}_{\bar{v}, i} = & \bar{a}_{\bar{w}^\prime, j} \\
\bar{a}_{\bar{v}^\prime, i} = &\bar{a}_{\bar{w}, j} \\
\end{align*}
so by transitivity of equality, $\bar{a}_{\bar{v}^\prime, i} = \bar{a}_{\bar{w}^\prime, j}.$  But now $\bar{v}^\prime, \bar{w}^\prime$ 
are witnesses to the collision which have a strictly lower number of $w-v$ crossings than $\bar{v}, \bar{w}$, whose number of crossings 
we had assumed to be minimal. This contradiction shows that, 
modulo the next step, 
we may indeed choose $\bar{v}$, $\bar{w}$ to have property $(\star)$. 

\br
\noindent\emph{Step 3: There is no nontrivial algebraicity in $J$.}
Suppose we are given a finite set $c_0 < \cdots < c_p$ from $J$, and $d \in J$ such that 
$c_r < d < c_{r+1}$. Then we claim the realizations of $\tpqf(d, \bar{c}, J)$ are dense in the interval 
$(c_r, c_{r+1})_J$. Why? This quantifier-free type is determined by the ordering $<$, the predicate which holds of $d$, 
and a partition of $[\{ c_0, \dots, c_p \}]^k$ into $X \cup Y$ so that 
the quantifer-free type specifies $\{ R(d,\tau) : \tau \in X \}$ and $\{ \neg R(d,\tau) : \tau \in Y \}$. 
Since the type already has a realization $d$, $X$ cannot contain any $(n+1)$-cliques. So this collection of 
conditions is consistent, and is realized densely in the interval $(c_r, c_{r+1})_J$ by $\aleph_0$-saturation. 

\br
\noindent\emph{Step 4: Using property $(\star)$ to contradict inconsistency.}
At this point in the proof, we have $\bar{v}, \bar{w}$, $i, j$ witnessing the contradiction, where $\bar{v}, \bar{w}$ 
satisfy property $(\star)$ from Step 2. We claim we can choose $\bar{z} \in \xr(J)$ such that 
\begin{equation}
\label{eq:55} \qftp(\bar{v}^\smallfrown\bar{w}, \bar{s}, J) = \qftp(\bar{v}^\smallfrown\bar{z}, \bar{s}, J) = 
\qftp(\bar{z}^\smallfrown\bar{w}, \bar{s}, J). 
\end{equation}
Why? 
Within each interval of $\bar{u}$, all the elements of $v$ are below all the elements of $w$ in the interval.  Suppose 
the number of elements of $w$ in the interval is $n_w$. 
Let $v_a$, $w_b$ be the maximal element of $v$ in the interval and the minimal element of $w$ in the interval, respectively. 
(Since $\bar{v}, \bar{w}$ were assumed to be in strictly increasing order, $v_{a-1}$, $w_{b+1}$ are the next largest and next smallest, 
respectively; the elements of $w$ in the interval are, in order, $w_{b} < \cdots < w_{b+n_w-1}$.) By induction on $\ell < n_w$, 
we apply the claim of Step 3 with $\bar{c}$ an enumeration of $\rn(\bar{v}) \cup \rn(\bar{w}) \setminus w_{b+\ell}$, 
and $d = w_{b+\ell}$, to choose a realization $w^\prime_{b+\ell}$ which is $<w_b$ (and necessarily strictly above $v_a$).   
Informally, one by one, we copy the elements of $w$
to the left.  Note that equation (\ref{eq:55}) implies a fortiori that $\qftp(\bar{z}, \bar{s}, J) = \qftp(\bar{t}, \bar{s}, J)$, which we called $\xr$. 
 
The result is $\bar{z}$ with the desired property, and note that within each interval of $\bar{u}$, all the elements of 
$\bar{v}$ there are below all the elements of $\bar{z}$ there which are in turn below all the elements of $\bar{w}$ there.

For the final contradiction, recall our original assumption (\ref{eq:coll}) which was an invariant of the quantifier-free type 
$\qftp(\bar{v}^\smallfrown \bar{w}, \bar{s}, J)$. Thus  
\begin{align*}
\bar{a}_{\bar{v}, i} = & \bar{a}_{\bar{w}, j} \\
\bar{a}_{\bar{z}, i} = & \bar{a}_{\bar{w}, j} \\
\bar{a}_{\bar{v}, i} = &\bar{a}_{\bar{z}, j} \\
\end{align*}
so by transitivity of equality, $\bar{a}_{\bar{z}, i} = \bar{a}_{\bar{z}, j}.$  This must be an invariant of 
$\qftp(\bar{z}, \bar{s}, J)$ but since this equals $\qftp(\bar{t}, \bar{s}, J)$, our original formula 
must have been inconsistent, contradicting (\ref{eq:cont0}). This completes the proof. 
\end{proof}

For the converse collection of claims, showing $\xc_{n,k}$-unsuperstability for the theories $T_{n,k}$, 
we begin with a warm-up, deriving a single instance of shearing for the case of $n=3, k=2$. We shall then upgrade this result 
to unsuperstability and to general $n,k$. However, in this warm-up case the notation is a bit simpler, 
and all the main ideas are represented. 

\begin{claim} \label{t32-claim}
The theory $T_{3,2}$ contains nontrivial $\xc_{3,2}$-shearing, coming from a formula which is a conjunction of positive instances of the 
edge relation.
\end{claim}

\begin{proof}
\tcb{Write $\mk = \mk_{3,2}$}.
We start with $I = I_{\xc_{3,2}} \in \mk$, which is countable, linearly ordered, each element is named by a different predicate, and there are no instances of $R$. 
Choose $\mathbf{t_0}, \mathbf{t_1}, \mathbf{t_2} \in I$ such that $I \models P_0(\mathbf{t_0}) \land P_1(\mathbf{t_1}) \land P_2(\mathbf{t_2})$, 
so it follows that 
$\mathbf{t_0}, \mathbf{t_1}, \mathbf{t_2}$ are distinct and (without loss of generality) $\mathbf{t_0} < \mathbf{t_1} < \mathbf{t_2}$. 
Because these elements are from $I$, note that their 
quantifier-free type specifies there are no instances of $R$.  We will write $\bar{\mathbf{t}} = 
\langle \mathbf{t_0}, \mathbf{t_1}, \mathbf{t_2} \rangle$. 

Consider some $\aleph_0$-saturated $J \in \mk$ with $I \subseteq J$.  Let $M \models T_{n,k}$ 
be fairly saturated. 
Choose a $\mk$-indiscernible sequence $\langle a_t : t \in J \rangle$ so that for any distinct 
$t_{i_0}, t_{i_1}, t_{i_2}$ from $J$, 
\[  M \models R(a_{t_{i_0}, t_{i_1}, t_{i_2}})~ \iff ~ J \models R(t_{i_0}, t_{i_1}, t_{i_2}).          \]
(The point is that the forbidden configuration never occurs on $J$, so we can build the map indexing the indiscernible sequence, say, by induction.) 
Now recalling $\bar{\mathbf{t}}$ from the first paragraph, consider in $M$ the formula 
\[ \vp(x,\bar{a}_{\bar{\mathbf{t}}}) = R(x,a_{\mathbf{t_0}}, a_{\mathbf{t_1}}) \land R(x,a_{\mathbf{t_0}}, a_{\mathbf{t_2}}) 
\land R(x,a_{\mathbf{t_1}}, a_{\mathbf{t_2}}). \] 
By our choice of $\bar{\mathbf{t}}$ and our choice of the indiscernible sequence, 
$M \models \neg R(a_{\mathbf{t_0}}, a_{\mathbf{t_1}}, a_{\mathbf{t_2}})$ (because the triple of indices come from $I$, which has no 
instances of $R$) so $\vp(x,\bar{a}_{\bar{\mathbf{t}}})$  
is a consistent formula of $M$. 
Since $J$ is $\aleph_0$-saturated for $\mk$, we may find $v_0$, $v_1$, $v_2$ in $J$ such that:\footnote{recalling that $\qftp$ in 
$J$ is determined by ordering, the predicates $P_q$, and the instances of $R$.}
\begin{align*} 
\tpqf({\mathbf{t_0}} ~^\smallfrown {\mathbf{t_1}} ~^\smallfrown \mathbf{t_2}, \emptyset, I) = & 
\tpqf({v_0} ~^\smallfrown {v_1} ~^\smallfrown \mathbf{t_2}, \emptyset, J) \\ 
= & \tpqf({v_0} ~^\smallfrown {\mathbf{t_1}} ~^\smallfrown v_2, \emptyset, J) \\
= & \tpqf({\mathbf{t_0}} ~^\smallfrown {v_1} ~^\smallfrown v_2, \emptyset, J) 
\end{align*} 
\underline{and} $J \models R(v_0, v_1, v_2)$. 
Then writing 
\begin{align*}
\bar{t}_{\{ 1,2 \}}=& {\mathbf{t_0}} ~^\smallfrown {v_1} ~^\smallfrown v_2 \\
\bar{t}_{\{0,2\}} =& {v_0} ~^\smallfrown {\mathbf{t_1}} ~^\smallfrown v_2 \\
\bar{t}_{\{0,1\}} =& {v_0} ~^\smallfrown {v_1} ~^\smallfrown \mathbf{t_2}
\end{align*}
we see that writing $\xr = \tpqf(\bar{\mathbf{t}}, \emptyset, I)$ we have that 
\[ \{ \vp(x,\bar{a}_{\bar{t^\prime}}): \bar{t}^\prime \in {\xr}(J) \} \]
is inconsistent, because the set of formulas 
\[ \{ \vp(x,\bar{a}_{\bar{t}_{\{1,2\}}}), \vp(x,\bar{a}_{\bar{t}_{\{0,2\}}}), \vp(x,\bar{a}_{\bar{t}_{\{0,1\}}}) \} \] 
is already inconsistent, recalling that $M \models R(a_{v_0}, a_{v_1}, a_{v_2})$ and 
\begin{align*}
\vp(x,\bar{a}_{\bar{t}_{\{1,2\}}}) \vdash& R(x,a_{v_1}, a_{v_2}) \\
\vp(x,\bar{a}_{\bar{t}_{\{0,2\}}}) \vdash& R(x,a_{v_0}, a_{v_2}) \\
\vp(x,\bar{a}_{\bar{t}_{\{0,1\}}}) \vdash& R(x,a_{v_0}, a_{v_1}).
\end{align*} 
This completes the warm-up. 
\end{proof}

\begin{claim} \label{claim-tnk}
For any $n>k\geq 2$, $T_{n,k}$ is $\xc_{n,k}$-unsuperstable. 
\end{claim}

\begin{proof}
Now we fix $n>k\geq 2$ and let us prove that $T_{n,k}$ is $\xc_{n,k}$-unsuperstable. 
For this proof, $I = I_{\xc_{n,k}}$ and $\mk = \mk_{\xc_{n,k}}$.   Recall the meaning of $n,k$: 
in $T_{n,k}$ the edge relation $R$ has arity $k+1$, and the forbidden configuration has arity 
$n+1$. 

To satisfy Definition \ref{m5a} in this case, \tcb{which recall describes a sequence indexed by $m$ witnessing 
unsuperstability}, it would suffice to show that whenever we are given 
a finite $m < \omega$ and sets $I_m, B_m, p_m$ such that $I_m$ is a finite subset of $I$, 
$B_m$ is a subset of the monster model of $T_{n,k}$, and $p_m$ is a partial non-algebraic type over $B_m$, 
we can find $I_{m+1}$, $B_{m+1}$, $p_{m+1}$ such that $I_{m+1}$ is finite and 
$I_m \subseteq I_{m+1} \subseteq I$, $B_{m+1} \supseteq B_m$ is a set of parameters in 
the monster model for $T_{n,k}$, and $p_{m+1}$ is a partial non-algebraic type over $B_{m+1}$ which extends 
$p_m$, \emph{and} such that $p_{m+1}$ $\xc_{n,k}$-divides over $B_m$. 

Suppose then that $I_m, B_m, p_m$ are given. 

Choose distinct $t_0, \dots, t_{n-1}$ from $I \setminus I_m$, so that $0 \leq i < j <n$ implies $t_i < t_j$ in $I$. 
Since we are in $I$, the element $t_i$ is named by the predicate $P_{t_i}$, and since $t_i \neq t_j$ their 
corresponding predicates are necessarily distinct. Also, since we are in $I$, there are necessarily no instances 
of $R$ on $\{ t_0, \dots, t_{n-1} \}$ in $I$.  Let $\bar{t} = \langle t_0, \dots, t_{n-1} \rangle$. Let 
$\xr = \tpqf(\bar{t}, {I_m}, I)$. 

Fix $J \supseteq I$ such that $J \in \mk$ and $J$ is $\aleph_0$-saturated. 
Choose a sequence $\langle a_j : j \in J \rangle$ which is $\mk$-indiscernible and such that 
\begin{equation}
\label{eq:rj} M \models R(a_{j_{i_0}}, \dots, a_{j_{i_{k}}}) ~ \iff ~ J \models R(j_{i_0}, \dots, j_{i_{k}}). 
\end{equation}
Now we will need some notation for two different but related things. First, recalling $\bar{t}$ from the previous paragraph, 
$\bar{a}_{\bar{t}}$ is now well defined (it is the tuple $\langle a_{j_i} : i < n \rangle$). For $u \in [n]^\ell$, let 
$\bar{t} \rstr u$ denote the $\ell$-tuple of elements of $J$ given by $\langle t_i : i \in u \rangle$, thus we let 
\[ \bar{a}_{\bar{t} \rstr u} \mbox{ denote the $\ell$-tuple of elements of $M$ given by } \langle a_{t_i} : i \in u \rangle. \]
Informally, we select an appropriate $\ell$-tuple from the sequence of length $n$. 
(We will mostly use $\ell = k$ or $\ell = k+1$.)
Second, if $\bar{v} = \langle v_0, \dots, v_{n-1} \rangle$ is any sequence from $\xr(J)$, and $u \in [n]^k$, let 
\begin{align*} \bar{v}_u  & \mbox{ denote the $n$-tuple of elements of $J$ given by } \langle t^\prime_i : i < n \rangle \\
& \mbox{ where } t^\prime_i = v_i \mbox{ if $i \in u$ } \mbox{ and } t^\prime_i = t_i \mbox{ if $i  \notin u$ }. 
\end{align*}
Informally, this is $\bar{v}$ where the $i$th element is replaced by $t_i$ if $i \notin u$. 

The next step is to choose $\bar{v} \in \xr(J)$ (and $\bar{v}$ will be fixed as this choice for the rest of the proof) 
so that:\footnote{Notice that condition (ii) is legal as the forbidden configuration has size $n+1$. 
Notice also that $u$ has size $k$, and the edge relation has arity $k+1$. Each $\bar{v}_u$ has 
`most,' i.e. all but $k$, of its elements from $\bar{t}$; this will be more noticeable when $n >> k$. So 
a priori, substituting in a few elements from $\bar{v}$ into the sequence $\bar{t}$ should not cause 
edges to appear.} 
\begin{align*} \mbox{(i)}& ~\tpqf(\bar{t}, I_m, I) = \tpqf(\bar{v}_u, I_m, J) \mbox{ for all $u \in [n]^k$ } \\
\mbox{ \emph{and}~ (ii)}& ~J \models R(\bar{v} \rstr w)~ \mbox{ for all $w \in [n]^{k+1}$.}  
\end{align*}
\begin{quotation}
\noindent For example, it would suffice to choose $\langle v_i : i < n \rangle$ by induction on $i$ as follows. 
Remember we assumed that $J \models t_0 < \cdots < t_n$, with $J \models P_{t_i}(t_i)$. 
Choose $v_i$ to be any element of $J$ satisfying the following two conditions: 
\begin{itemize}
\item[(a)] $J \models t_i < v_i < t_{i+1}$ and $P_{t_i}(v_i)$.  
\item[(b)]  partition $ [ \{ t_0, \dots, t_n \} \cup \{ v_j : j < i \} ]^k$ into two sets: 
\\ $Y = [ \{ v_j : j < i \} ]^k$, and $X$ is the complement of $Y$ (so members of $X$ contain at least one $t_\ell$). 
Ask that $t_i$ satisfies: $\tau \cup \{ t_i \} \in R^J$ for $\tau \in Y$, and 
$\tau \cup \{ t_i \} \notin R^J$ for $\tau \in X$. 
\end{itemize}
At each step, this choice is possible because these conditions will not produce 
an $(n+1)$-clique, and $J$ is $\aleph_0$-saturated.  Since quantifier-free type in $J$ is determined by the ordering 
$<$, the predicates $P_\ell$, and the edge relation $R$, this suffices. 
\end{quotation}
Now in $M$, consider the formula
\[ \vp(x,\bar{a}_{\bar{t}}) = \bigwedge_{u \in [n]^k} R(x,\bar{a}_{\bar{t} \rstr u}). \]
Since there are no instances of $R$ on $\bar{t}$ in $J$, there are correspondingly no instances of $R$ on 
$\{ a_{t_i} : i < n \}$ in $M$ by equation (\ref{eq:rj}), so this formula is consistent. 
However, the set of formulas 
\[  \{ \vp(x,\bar{a}_{\bar{t}^\prime}) : \bar{t}^\prime \in \xr(J) \} \]
is inconsistent, because by condition (a) above, it includes the set of formulas 
\[ \{ \vp(x,\bar{a}_{\bar{v}_u}) : u \in [n]^k \} \]
which is inconsistent as for each $u \in [n]^k$, $\vp(x,\bar{a}_{\bar{v}_u}) \vdash R(x,\bar{a}_{\bar{v} \rstr u})$, and also  
$M \models R(\bar{a}_{\bar{v} \rstr w})$ for each $w \in [n]^{k+1}$ by condition (ii) and equation (\ref{eq:rj}); together, these would 
give us an $(n+1)$-clique which is forbidden.   This is our desired instance of $\xc_{n,k}$-shearing.

To finish, let $I_{m+1} = I_m \cup \{ t_i : i < n \}$. Let $B_{m+1}$ be an elementary submodel of $M$ which contains 
$B_m \cup \{ a_t : t \in J \}$.  Let $p_{m+1} = p_m \cup \{ \vp(x,\bar{a}_{\bar{t}}) \}$.  (Note that in our theory $T_{n,k}$, 
$p_{m+1}$ is consistent simply because the parameters from $\{ a_{t_i} : i < n \}$ are disjoint to $B_m$.)

This completes the proof. 
\end{proof}

\begin{theorem} \label{t:disting} 
For each $n>k\geq 2$, there is a countable context $\xc$ such that $\trg$ is $\xc$-superstable but 
$T_{n,k}$ is $\xc$-unsuperstable.
\end{theorem}

\begin{proof}  
Use $\xc_{n,k}$ from Definition \ref{d:xcnk}.  By Claim \ref{trg-ss}, $\trg$ is $\xc$-superstable. By Claim \ref{claim-tnk}, 
$T_{n,k}$ is $\xc_{n,k}$-unsuperstable. 
\end{proof}

\br

\section{Some questions and future directions}

\setcounter{theoremcounter}{0}
We also refer the reader to the open problems section of \cite{MiSh:1149}.  

\begin{problem} \label{p:ocirc} 
Characterize the theories which have the same shearing as the random graph, that is, the theories $T$ which satisfy: for 
every countable context $\xc$, $T$ is $\xc$-unsuperstable if and only if $\xc$ has property $\ocirc$. 
\end{problem}

\begin{rmk}  
By inspection, the theories $T_\xm$ from \cite{MiSh:1167} appear to have the same shearing as the random graph. This suggests the 
class from Problem \ref{p:ocirc} is interesting, spanning classes in Keisler's order and the interpretability order. 
\end{rmk}

\begin{problem}  
Investigate the relation of shearing and dividing in ``basic'' non-simple theories such as $T_{\operatorname{feq}}$, for example, 
investigate whether it is possible to characterize there the shearing which does not come from dividing.  
\end{problem} 

\begin{problem} 
The equivalence of shearing and dividing in stable theories is explained by the phenomenon of ``weak definitions,'' see explanation and 
references in the proof of \ref{c:unstable} above.  The existence of weak definitions is a property of $(T, \vp)$ and of a context.  
Thus the methods of this paper open up the possibility of developing a theory of the 
\emph{relative strength} of weak definitions, as contexts vary, in simple unstable theories. 
\end{problem}

\begin{rmk}  
The methods of this paper also represent an interesting interaction between elementary classes $($the theories $T$$)$ 
and a priori non-elementary classes $($the index model classes $\mk$$)$.  
\end{rmk}

\begin{rmk}
Since, by the proofs above, shearing is strictly weaker than dividing, 
it necessarily fails some of the axioms for independence relations in simple theories \cite{kp};  
it may be useful to sort out which hold and which do not. 
\end{rmk}

\begin{qst}  
Is there an analogous  
axiomatic characterization of shearing?
\end{qst}

\begin{rmk} 
The statement and proof of \ref{c:unstable} implicitly raise the analogous question for NIP theories, 
i.e., it may be worth while to explicitly sort out what can be said about shearing for 
formulas {without} the independence property. 
\end{rmk}

\end{document}